\documentclass[a4paper,10pt]{article}
\usepackage{amsmath,amsthm,amssymb, amsfonts, latexsym,epic,bbm,comment,mathrsfs}
\usepackage{graphicx,enumerate,stmaryrd, xcolor, color}
\usepackage[all,cmtip]{xy}
\usepackage{a4wide}
\usepackage{mathdots}
\usepackage{mathtools}
\usepackage{extarrows}
\usepackage{arydshln}
\usepackage{float,subfig,overpic}
\usepackage[verbose,colorlinks=true,linktocpage=true,linkcolor=blue,citecolor=blue]{hyperref}
\usepackage{pgfplots}
\usepackage{tikz}
\usepackage{indentfirst} 

\usetikzlibrary{patterns,arrows,positioning,calc,fadings,shapes,decorations.markings}

\theoremstyle{definition}
\newtheorem{definition}{Definition}[section]

\theoremstyle{plain}
\newtheorem{theorem}[definition]{Theorem}
\newtheorem{proposition}[definition]{Proposition}
\newtheorem{lemma}[definition]{Lemma}
\newtheorem{keylemma}[definition]{Key Lemma}
\newtheorem{corollary}[definition]{Corollary}
\newtheorem{example}[definition]{Example}
\theoremstyle{remark}
\newtheorem{remark}[definition]{Remark}

\numberwithin{equation}{section}

\newcommand{\mf}{\mathfrak}
\newcommand{\C}{\mathbb C}

\newcommand{\oa}{{\bar 0}}
\newcommand{\ob}{{\bar 1}}
\newcommand{\vare}{\epsilon} 
\def\gl{\mathfrak{gl}}
\newcommand{\g}{\mathfrak{g}}

\newcommand{\h}{\mathfrak{h}}

\newcommand{\Z}{{\mathbb Z}}

\newcommand{\RNum}[1]{\uppercase\expandafter{\romannumeral #1\relax}}
\allowdisplaybreaks[1]

\begin{document}

\title{The Schur-Weyl duality and Invariants for classical Lie superalgebras}
\author{Yang Luo ${}^1$ and Yongjie Wang ${}^{2}$}
\maketitle

\begin{center}
\footnotesize
\begin{itemize}
\item[1] College of Science, National University of Defense Technology, Changsha, Hunan, 410073, China.
\item[2] School of Mathematics, Hefei University of Technology, Hefei, Anhui, 230009, China.
\end{itemize}
\end{center}
\begin{abstract}
In this article, we provide a comprehensive characterization of invariants of classical Lie superalgebras from the super-analog of the Schur-Weyl duality in a unified way. We establish $\mathfrak{g}$-invariants of the tensor algebra $T(\mathfrak{g})$, the supersymmetric algebra $S(\mathfrak{g})$, and the universal enveloping algebra $\mathrm{U}(\mathfrak{g})$ of a classical Lie superalgebra $\mathfrak{g}$ corresponding to every element in centralizer algebras and their relationship under supersymmetrization. As a byproduct, we prove that the restriction on $T(\mathfrak{g})^{\mathfrak{g}}$ of the projection from $T(\mathfrak{g})$ to $\mathrm{U}(\mathfrak{g})$ is surjective, which enables us to determine the generators of the center $\mathcal{Z}(\mathfrak{g})$ except for $\mathfrak{g}=\mathfrak{osp}_{2m|2n}$. Additionally, we present an alternative algebraic proof of the triviality of $\mathcal{Z}(\mathfrak{p}_n)$. The key ingredient involves a technique lemma related to the symmetric group and Brauer diagrams. 
\bigskip

\bigskip

\end{abstract}

\tableofcontents

\section{Introduction}
\label{sec:intr}

Schur-Weyl duality is one of the most significant and influential topics in representation theory. It not only reveals numerous hidden connections between the representation theories of seemingly unrelated algebras, but also serves as a powerful tool for studying the structure and representation theory of the corresponding Lie (super)algebras. For example, the symmetric group algebra  $\mathbb{C}\mathfrak{S}_d$ appears as the centralizer of the general linear Lie (super)algebra action on the tensor space $V^{\otimes d}$, where $V$ is the natural representation of $\mathfrak{gl}_n$ or $\mathfrak{gl}_{m|n}$.  Since then, there have been various attempts to depict the centralizer algebra of natural representations of other Lie (super)algebras. The Brauer algebras with specific parameters can be realized as the centralizer algebra of the orthogonal, the symplectic Lie algebra, or the  ortho-symplectic Lie superalgebra acting on the tensor product of the corresponding natural representation. Similarly,  the Hecke-Clifford algebra and the Periplectic Brauer algebra are the centralizer algebras of the queer Lie superalgebra and the periplectic Lie superalgebra, respectively. We refer to \cite{CW12, Co18, ES16, LZ12, LZ15, M18, Mo03, Mu12, Se84}  and references therein for more results on the Schur-Weyl dualities. 

There is a significant correlation between the Schur-Weyl duality and the center of the universal enveloping algebra, which is primarily manifested in two aspects. Firstly, the duality implies that the centers $\mathcal{Z}(\mathfrak{gl}_n)$  and $\mathcal{Z}(\mathbb{C}\mathfrak{S}_d)$ of $\mathrm{U}(\mathfrak{gl}_n)$ and $\mathbb{C}\mathfrak{S}_d$, respectively, have the same images in the automorphisms of the corresponding tensor spaces. Kerov and Olshanski explicitly described elements of $\mathcal{Z}(\mathfrak{gl}_n)$ that act identically on the tensor space as a given element of $\mathcal{Z}(\mathbb{C}\mathfrak{S}_d)$, see \cite{KO94} for more details.  
This result has recently been generalized to Lie superalgebras $\mathrm{U}(\mathfrak{gl}_{m|n})$ and $\mathrm{U}(\mathfrak{q}_n)$ by Borodin and Rozhkovskaya in \cite{BR22}. 

Secondly, the duality can be utilized to obtain natural constructions of families of Casimir elements for the classical Lie algebras.  In the case of the general linear Lie algebra, these elements are known as quantum immanants, as mentioned in references \cite{O96, OO98}, and they form a linear basis of the center of the universal enveloping algebra $\mathrm{U}(\mathfrak{gl}_n)$. Iorgov, Molev, and Ragoucy employed the Schur-Weyl duality to create Casimir elements for classical Lie algebras and demonstrated that these elements are algebraically independent generators of the center by computing the image of the Harish-Chandra homomorphism in their research \cite{IMR13}. These generators can also be expressed as certain non-commutative determinants and permanents.

Furthermore,  Molev and Retakh utilized the techniques of quasideterminants in their publication \cite{MR04} to derive families of Casimir elements for the general Lie superalgebra $\mathfrak{gl}_{m|n}$. The images of these elements under the Harish-Chandra homomorphism are elementary, complete, and power sums supersymmetric functions.
 Consequently, they obtained the several families of generators for the center $\mathcal{Z}(\mathfrak{gl}_{m|n})$. Recently, Lehrer and Zhang have system investigated on the invariant theory of Lie (super)algebras \cite{LZ12, LZ15, LZ17a, LZ17b, LZ21, LZ24}, also in collaboration with Deligne \cite{DLZ18}. Specifically, they have obtained the super Paffian of the ortho-symplectic Lie superalgebras, which represents a significant advancement beyond our current techniques.

In this paper, our purpose is to establish explicit invariants of classical Lie superalgebras $\mathfrak{g}$ arising from the super-analog of the Schur-Weyl duality and investigate the explicit relationships between the $\mathfrak{g}$-invariant subalgebras $T(\mathfrak{g})^{\mathfrak{g}}$, $S(\mathfrak{g})^{\mathfrak{g}}$, and $\mathrm{U}(\mathfrak{g})^{\mathfrak{g}}$. In the following diagram, the map $\eta$ is surjective and $\psi$ is bijective. However, it is unknown whether $\eta'$ is surjective or not due to the lack of commutativity.

\begin{align*}
 \xymatrix{
T(\mathfrak{g})^{\mathfrak{g}}\ar[r]^{\eta}\ar[dr]_{\eta'}&S(\mathfrak{g})^{\mathfrak{g}}\ar@{->}[d]^{\psi} \\
&\mathrm{U}(\mathfrak{g})^{\mathfrak{g}}
.}
\end{align*}

We provide an affirmative answer to this question for Lie superalgebras $\mathfrak{gl}_{m|n}$, $\mathfrak{osp}_{2m+1|2n}$, $\mathfrak{p}_n$, and $\mathfrak{q}_n$ by utilizing a new method to obtain central elements of the universal enveloping algebra from the super version of Schur-Weyl duality. To be more explicitly, let $(V,\iota)$ be the natural representation of a classical Lie superalgebra $\mathfrak{g}$, then there is a split surjective homomorphism $\tilde{\pi}$ from $\mathrm{End}(V)$ to $\mathfrak{g}$, see Proposition  \ref{le:split}, which induces a surjective linear map $\pi$ from $\mathop{\bigoplus}\limits_{k\geqslant 0}\left[\mathrm{End}(V)^{\otimes k}\right]^{\mathfrak{g}}$ to $T(\mathfrak{g})^{\mathfrak{g}}$. Note that $\mathop{\bigoplus}\limits_{k\geqslant 0}\left[\mathrm{End}(V)^{\otimes k}\right]^{\mathfrak{g}}$ is isomorphic to $\mathop{\bigoplus}\limits_{k\geqslant 0}\mathrm{End}_{\mathfrak{g}}(V^{\otimes k})$ in a canonical manner, and the latter is isomorphic to  the direct sum of the centralizer algebra through the corresponding Schur-Weyl duality. 
In summary, we have the following diagram:  
 $$
\xymatrix{\bigoplus\limits_{k\geqslant0}\mathcal{A}_k\ar@{>}[r]^(.37){\Psi} &\bigoplus\limits_{k\geqslant0}\mathrm{End}_{\mathfrak{g}}(V^{\otimes k})\ar@{->>}[r]^(.46){\Omega}&\bigoplus\limits_{k\geqslant0}\left[\mathrm{End}(V)^{\otimes k}\right]^{\mathfrak{g}}\ar@{-->>}[r]^(.62){{\pi}}&
T(\mathfrak{g})^{\mathfrak{g}}\ar@{->>}[r]^{\eta}\ar[dr]^{\eta'}&S(\mathfrak{g})^{\mathfrak{g}}\ar@{->>}[d]^{\psi}\\
&&&&\mathcal{Z}(\mathfrak{g}),
}
$$
where  $\Psi=\bigoplus\limits_{k\geqslant0}\Psi_k$, $\pi=\bigoplus\limits_{k\geqslant0}\tilde{\pi}^{\otimes k}$, and  $\Omega$ is the direct sum of the canonical map on each summand.

The split surjective homomorphism $\tilde{\pi}$ builds a bridge, which enables us to convert all elements of centralizer algebras into $\mathfrak{g}$-invariants of $T(\mathfrak{g})$, $S(\mathfrak{g})$ and $\mathrm{U}(\mathfrak{g})$. By comparing the $\mathfrak{g}$-invariants of $S(\mathfrak{g})$ under the supersymmetrization map $\psi$ with the $\mathfrak{g}$-invariants of $\mathrm{U}(\mathfrak{g})$, we prove that $\eta'$ is surjective if $\mathfrak{g}\neq \mathfrak{osp}_{2m|2n}$, see Theorem \ref{eta'surjective}.

We present the generators, which are constructed by certain $k$-cycles of the centralizer algebras, of the center $\mathcal{Z}(\mathfrak{g})$ of the universal enveloping algebra $\mathrm{U}(\mathfrak{g})$ for the general linear Lie superalgebra $\mathfrak{gl}_{m|n}$, ortho-symplectic Lie superalgebra $\mathfrak{osp}_{2m+1|2n}$, and queer Lie superalgebra $\mathfrak{q}_n$, as well as partial generators for the ortho-symplectic Lie superalgebra $\mathfrak{osp}_{2m|2n}$ case by calculating the image of invariants under the Harish-Chandra homomorphisms. These results are consistent with the relevant research conducted in \cite{IMR13, Jar79, Jar83, M18,  NS06, Sc83c, Se83}.

We observe that Scheunert obtained the Casimir elements of a Lie superalgebra through the investigation of invariant supersymmetric multilinear forms on its coadjoint module, see \cite{Sc87}. Specifically, by employing Lie supergroup and algebraic geometry techniques, he deduced that the center $\mathcal{Z}(\mathfrak{p}_n)$ is trivial in \cite{Sc87}. The center of the other $P$-type Lie superalgebras has been investigated by Georelik in \cite{Go01}. In this article, we present a purely algebraic proof of the triviality. To accomplish this, we simplify the problem to one,  see Key Lemma \ref{symmetricgroup}, which only involves the symmetric group. Fortunately, we discover that this issue can be solved by investigating the properties of the Brauer diagrams.

The paper is organized as follows: In Subsection \ref{Liesuperalgebra}, we set up the fundamental facts related to the Lie superalgebras $\mathfrak{gl}_{m|n}, \mathfrak{osp}_{m|2n}$, $\mathfrak{q}_n$ and $\mathfrak{p}_n$.  In Subsections \ref{sect:harish-chandra} and \ref{sect:schurweyl}, we provide background materials on the Harish-Chandra homomorphism and the Schur-Weyl duality of the classical Lie superalgebras,  along with other general technical results and preparatory tools that are to be used in the sequel. 
In Section \ref{se:Centers}, we focus on the investigation on the relationship between the tensor, the supersymmetric algebra and the universal enveloping algebra of a Lie superalgebra $\mathfrak{g}$, see Theorem \ref{etaeta'}, Proposition \ref{osprelation}. We also obtain the Gelfand invariants for classical Lie superalgebra, by extending earlier work by Molev \cite[Chapter 5]{M18}, see Theorems \ref{Gelfandgl} and \ref{Gelfandosp}. Also in this section, we deduce the 
generators of the center $\mathcal{Z}(\mathfrak{q}_n)$, and the triviality of the center $\mathcal{Z}(\mathfrak{p}_n)$, which coincide with Sergeev's \cite{Se83} and Scheunert's work\cite{Sc87}, respectively. 
\medskip

\noindent{\bf Notations and terminologies:}

Throughout this paper, we will use the standard notations $\mathbb{Z}$, $\mathbb{Z}_+$ and $\mathbb{N}$ that represent the sets of integers, non-negative integers and positive integers, respectively. Denote by $\mathcal{Z}(\mathfrak{g})$ the center of the universal enveloping (super)algebra of a Lie (super)algebra $\mathfrak{g}$. The Kronecker delta $\delta_{ij}$ is equal to $1$ if $i=j$ and $0$ otherwise.

We write $\mathbb{Z}_2=\left\{\bar{0},\bar{1}\right\}$. For a homogeneous element $x$ of an associative or Lie superalgebra, we use $|x|$ to denote the degree of $x$ with respect to the $\mathbb{Z}_2$-grading. Throughout the paper, when we write $|x|$ for an element $x$,  we will always assume that $x$ is a homogeneous element and automatically extend the relevant formulas by linearity (whenever applicable). All modules of Lie superalgebras and quantum superalgebras are assumed to be $\mathbb{Z}_2$-graded. The tensor product of two superalgebras $A$ and $B$ carries a superalgebra structure by 
$$(a_1 \otimes b_1)\cdot(a_2\otimes b_2) =(-1)^{|a_2||b_1|}a_1a_2\otimes b_1b_2,$$
for homogeneous elements $a_i\in A, b_i\in B$ with $i=1,2$.

Let $\mathbb{C}^{N}$ be a $N$-dimensional superspace with homogeneous basis $e_1,e_2,\cdots,e_N$. We briefly denote the degree of $v_k$ with $|k|$ for all $k$. Let $\mathrm{End}(\mathbb{C}^N)$ be the endomorphism algebra of superspace $\mathbb{C}^{N}$ and denote by $e_{ij}$ the endomorphism $e_{ij}e_k=\delta_{jk}e_i$ for all $k$. 

We call $X=[X_{ij}]$ with entries in an associative superalgebra $\mathcal{A}$ as $N\otimes N$ supermatrix, and it can be regarded as the element 
\begin{align}\label{tensornotation}
X=\sum\limits_{i,j=1}^N(-1)^{|i||j|+|j|}e_{ij}\otimes X_{ij}\in\mathrm{End}(\mathbb{C}^N)\otimes\mathcal{A}.
\end{align}
We will need tensor product superalgebras of the form 
$\mathrm{End}(\mathbb{C}^N)^{\otimes m}\otimes \mathcal{A}$. For any $1\leqslant a\leqslant m$, we will denote by $X_a$ the element associated withe $a$-th  copy of $\mathrm{End}(\mathbb{C}^N)$ so that 
$$X_a=\sum\limits_{i,j=1}^N(-1)^{|i||j|+|j|} 1^{\otimes (a-1)}\otimes  e_{ij}\otimes 1^{\otimes (m-1)}\otimes X_{ij}\in\mathrm{End}(\mathbb{C}^N)^{\otimes m}\otimes \mathcal{A}.$$ 
The supertrace map $\mathrm{Str}\colon\mathrm{End}(\mathbb{C}^N)\longrightarrow\mathbb{C}$ is defined by $e_{ij}\mapsto (-1)^{|i|}\delta_{ij}.$ 
Furthermore, for any $a\in\{1,\ldots,m\}$ the partial supertrace $\mathrm{Str}_{a}$ will be defined as follows:
$$\mathrm{Str}_{a}:\ \mathrm{End}(\mathbb{C}^N)^{\otimes m}\longrightarrow\mathrm{End}(\mathbb{C}^N)^{\otimes (m-1)},$$
which takes the supertrace on the the $a$-th copy of $\mathrm{End}(\mathbb{C}^N)$ and acts the identity map on all the remaining copies. The full supertrace $\mathrm{Str}_{1,\cdots,m}$ is the composition $\mathrm{Str}_{1}\circ\cdots\circ\mathrm{Str}_{m}$.

{\bf  Acknowledgment}.  The first author is supported by  NSF of China (Grants No. 12401037) and NUDT (No. 202401-YJRC-XX-002). The second author is supported by NSF of China (Grants Nos. 12071026, 1240125) and Anhui Provincial Natural Science Foundation 2308085MA01. We are very grateful to Prof. S.-J. Cheng, Prof. A. Molev, Prof. R. B. Zhang and Dr. W. Liu for their helpful discussions and  comments.

\section{Preliminaries}\label{Liequantum}
\subsection{Lie superalgebras}\label{Liesuperalgebra}
In this section, we will review some essential facts regarding Lie superalgebras including root systems, matrix realizations, the enveloping superalgebras and the Harish-Chandra homomorhisms. This will help us establish our notations. For more comprehensive information on the theory of Lie superalgebras, we recommend referring to the books \cite{CW12} and \cite{Mu12}.
Recall that a finite-dimensional Lie superalgebra $\mathfrak{g}=\mathfrak{g}_{\bar{0}}\oplus\mathfrak{g}_{\bar{1}}$ over the complex field is referred to as \textit{classical}, or \textit{quasi-reductive}, if the restriction of the adjoint representation of $\mathfrak{g}$ to the even part $\mathfrak{g}_{\bar{0}}$ is semisimple. In the present paper, we are mainly interested in the following series of classical Lie superalgebras:
\begin{align}\label{eq:glist}
\g= \gl_{m|n},\ \mf{osp}_{m|2n},\ \mathfrak{q}_{n}  \text{ and }  \mathfrak{p}_{n}, \text{  for $m,n\in \Z_{\geq 0}$}.
\end{align}
\subsubsection{The general linear Lie superalgebra $\gl_{m|n}$}
The Lie superalgebra $\gl_{m|n}$ can be realized as
the space of $(m+n)$ by $(m+n)$ \mbox{complex matrices}  
\begin{align*}
	  \left( \begin{array}{cc} A & B\\
		C & D\\
	\end{array} \right),
\end{align*} where $A\in \C^{m\times m}, B\in \C^{m\times n}, C\in \C^{n\times m}, D\in \C^{n\times n}$, with the Lie bracket
given by the super commutator; see, \cite[Section 1.1.2]{CW12}, \cite[Chapter 2]{Mu12} or \cite{Ka77} for more details. 
 
 Denote by  $E_{ij}$   the elementary matrix in $\gl_{m|n}$ with $(i,j)$-entry $1$ and other entries $0$, for $1\leqslant i,j\leqslant m+n$.  The parity of $E_{ij}$ is $|i|+|j|$, where $|i|=
 \bar{0}$ if  $1\leqslant i\leqslant m$, or $\bar{1}$, otherwise.
 We fix the Cartan subalgebra $\h:=\bigoplus_{i=1}^{m+n}\C E_{ii}$ consisting of all diagonal matrices. Denote by $h_i=E_{ii}$ and $h'_j=E_{m+j,m+j}$ for all $1\leqslant i\leqslant m$ and $1\leqslant j\leqslant n$. Let 
$\vare_i$ and $\delta_j$ denote the standard dual basis elements for $\h^\ast$  determined by $\vare_j(h_{i}) =\delta_{ij}$,  $\delta_p(h'_{q}) = \delta_{pq}$, and $\vare_j(h'_{q}) = \delta_p(h_{i}) =0$, for any $1\leqslant i,j\leqslant m$ and $1\leqslant p,q\leqslant n.$
The corresponding sets of even and odd roots are given by
\begin{align*}
	&\Phi_\oa:= \{\vare_i -\vare_j,~\delta_{p}-\delta_q|~1\leqslant i\neq  j\leqslant m,~1\leqslant p\neq  q\leqslant n\},\\
	&\Phi_\ob:=\{\pm(\vare_i-\delta_p)|~1\leqslant i\leqslant m,~1\leqslant p\leqslant n\},
\end{align*} respectively.
We fix the positive system $\Phi^+=\Phi_\oa^+\cup\Phi_\ob^+$ as follows:  
\begin{align*}
		&\Phi_\oa^+:= \{\vare_i -\vare_j,~\delta_{p}-\delta_q|~1\leqslant i< j\leqslant m,~1\leqslant p< q\leqslant n\},\\
	 &\Phi_\ob^+:=\{\vare_i-\delta_j|~1\leqslant i\leqslant m,~1\leqslant j\leqslant n\}. 
\end{align*}
Then the weyl vector $\rho=\frac{1}{2}\sum\limits_{\alpha\in\Phi_\oa^+}\alpha-\frac{1}{2}\sum\limits_{\alpha\in\Phi_\ob^+}\alpha$. The set $\Phi^-$ of all negative roots is given by  $\Phi^-:= -\Phi^+$. 

We define the bilinear form  $\langle\cdot,\cdot\rangle:\h^\ast\times \h^\ast\rightarrow \C$  on $\h^\ast$ as the one induced by the standard super-trace form, that is, $\langle\vare_i,\vare_j\rangle =\delta_{ij}$ and $\langle\delta_p,\delta_q\rangle =-\delta_{pq}$ and $\langle\vare_i,\delta_p\rangle=0$, for all $1\leqslant i,j\leqslant m$ and $1\leqslant p,q\leqslant n$.

\subsubsection{The queer Lie superalgebra $\mathfrak{q}_{n}$} \label{queer} 
Let $V=V_{\bar{0}}\oplus V_{\bar{1}}$ be a superspace with $\mathrm{dim}V_{\bar{0}}=\mathrm{dim}V_{\bar{1}}=n$. Choose $\mathcal{P}\in\mathrm{End}_{\bar{1}}$ such that $\mathcal{P}^2=I_{V}$. The subspace
\begin{align*}
\mathfrak{q}_n=\{T\in \mathrm{End}(V)|[T,\mathcal{P}]=0\}
\end{align*}
is a subalgebra of $\mathfrak{gl}(V)$ called the \textit{queer Lie superalgebra}. There exists a basis $B=B_0\cup B_1$ such that $B_0=\{e_1,\cdots,e_n\}$ is a basis for $V_{\bar{0}}$ and $B_1=\{e_{-1},\cdots,e_{-n}\}$ is a basis for $V_{\bar{1}}$, and
\begin{align}\label{eq:cliffact}
\mathcal{P}e_i=-\sqrt{-1}e_{-i},\quad \mathcal{P}e_{-i}=\sqrt{-1}e_i ,
\end{align}
for all $i=1,2,\cdots,n$. Then $\mathfrak{q}_{n}$ is spanned by the elements $H_{ij}=e_{ij}+e_{-i,-j}$ with the parity $|i|+|j|$.

 The Lie superalgebra $\mathfrak{q}_{n}$ can be realized as a subalgebra inside $\gl_{n|n}$, and it is   
 \begin{align}
 	\mathfrak{q}_{n}=
 	\left\{\left. \left( \begin{array}{cc} A & B\\
 		B & A\\
 	\end{array} \right) \right| ~ A,B\in \mathfrak{gl}_{n}\right\}.  \label{eq::max::qn}
 \end{align}  
 Fix the Cartan subalgebra  $\h= \bigoplus_{i=1}^n (\C h_{i}+\mathbb{C} h_i')$ consisting of diagonal matrices, where $h_i=e_{ii}+e_{n+i,n+i}$ and $h'_i=e_{i,n+i}+e_{n+i,i}$ for all $1\leqslant i\leqslant n$. Let $\{\vare_1, \vare_2, \ldots, \vare_n\}$ denote the basis for $\h^\ast_{\bar{0}}$ dual to the standard basis $h_{1},h_{2},\ldots, h_{n}$. The sets of even and odd roots are given by 
 \begin{align*}
 	&\Phi_\oa =\Phi_\ob = \{\vare_i-\vare_j|~1\leqslant i\neq  j \leqslant n\}, \end{align*}
the set $\Phi^+$ of positive roots and the set $\Phi^-$ of negative roots are listed as follows:
\begin{align*}
&\Phi^+:=\{\vare_i-\vare_j|~1\leqslant i< j \leqslant n\},\quad
\Phi^-:=\{-\vare_i+\vare_j|~1\leqslant i< j \leqslant n\}.
\end{align*} 

Then the weyl vector $\rho=\frac{1}{2}\sum\limits_{\alpha\in\Phi_\oa^+}\alpha-\frac{1}{2}\sum\limits_{\alpha\in\Phi_\ob^+}\alpha=0$.

Let $\mathfrak{q}_n^{\perp}$ be the set of all block matrices of the form 
\begin{align}
\left( \begin{array}{cc} A & B\\
 		-B & -A\\
 	\end{array} \right),\end{align}
where $A, B$ are $n\times n$ matrices. It is a complementary subspace of $\mathfrak{q}_{n}$ in $\mathrm{End}(V)$.
 
Moreover, $\mathrm{End}(V)=\mathfrak{q}_{n}\oplus \mathfrak{q}_{n}^{\perp}$ as adjoint $\mathfrak{q}_{n}$-modules.  Denote by $\tilde{\pi}$ the natural projection associated with the direct sum decomposition. Note that $\mathfrak{q}_{n}$ is closed under matrix multiplication, and therefore $\tilde{\pi}$ is a module algebra homomorphism from $\mathrm{End}(V)$ to $\mathfrak{q}_{n}$.

Furthermore, the endomorphism algebra $\mathrm{End}(V)$ is equal to the direct sum of $\mathfrak{q}_{n}$ and its orthogonal complement $\mathfrak{q}_{n}^{\perp}$ as adjoint $\mathfrak{q}_{n}$-modules. Let $\tilde{\pi}$ denote the natural projection associated with this direct sum decomposition. It is important to note that $\mathfrak{q}_{n}$ is closed under matrix multiplication, thus $\tilde{\pi}$ is a module algebra homomorphism from $\mathrm{End}(V)$ to $\mathfrak{q}_{n}$.

\subsubsection{The ortho-symplectic Lie superalgebra $\mf{osp}_{m|2n}$}\label{orth}  Let $V=V_{\bar{0}}\oplus V_{\bar{1}}$ be a superspace with $\mathrm{dim}V_{\bar{0}}=m$ and $\mathrm{dim}V_{\bar{1}}=2n$. Suppose $B(\cdot,\cdot)$ is a non-degenerate supersymmetric even bilinear form on $V$. This means that, the restriction of $B$ to $V_{\bar{0}}$ is symmetric, the restriction of $B$ to $V_{\bar{1}}$ is skew-symmetric, and $V_{\bar{0}}$ and $V_{\bar{1}}$ are orthogonal under $B$. The ortho-symplectic Lie superalgebra $\mathfrak{osp}_{m|2n}$ is the Lie sub-superalgebra of $\mathrm{End}(V)$ that preserves the form $B$. More precisely, 
$\mathfrak{osp}_{m|2n}=(\mathfrak{osp}_{m|2n})_{\bar{0}}\oplus (\mathfrak{osp}_{m|2n})_{\bar{1}}, \text{ where }$
$$
(\mathfrak{osp}_{m|2n})_s=\left\{a\in \mathrm{End}(V)\left| B(ax,y)=-(-1)^{s\cdot|x|}B(x,ay)\right.~~\forall x,y\in V\right\},\quad s\in\mathbb{Z}_2.
$$

Choose an appropriate basis $\{e_i\}_{1\leqslant i\leqslant m+n}$ such that $e_i\in V_{\bar{0}}$ if and only if $n<i\leqslant m+n$ and the otherwise are odd, then the Gram matrix $T$ of form $B$ can be written as
\begin{equation*}
T=\begin{array}{c@{\hspace{-5pt}}l}
\left(
\begin{array}{ccc|ccc|ccc}
& &  &  &   & &&&1\\
& &  &  &   & && \iddots &\\
& &  &  &   & &1&&\\ \hline   
& &  &  &   & 1&&&\\
& &  &  &  \iddots  & &&&\\
& &  &  1&   & &&&\\ \hline

& & -1 &  &   & &&&\\ 
& \iddots&  &  &   & &&&\\ 
-1& &  &  &   & &&&\\    
\end{array}     
\right)    
&
\begin{array}{l}
\left.\rule{0mm}{7.4mm}\right\}n\\
\\
\left.\rule{0mm}{6.5mm}\right\}m\\
\\
\left.\rule{0mm}{6.8mm}\right\}n\\
\end{array}
\\[-5pt]
\begin{array}{ccc}
\underbrace{\rule{22mm}{0mm}}_n
\underbrace{\rule{17.6mm}{0mm}}_m
\underbrace{\rule{17mm}{0mm}}_n
\end{array}
\end{array}.
\end{equation*}
For each $1\leqslant i\leqslant m+2n$, define $i'=m+2n+1-i$. Set
\begin{align*}
\varepsilon_i=\begin{cases}
1,&\text{for } 1\leqslant i\leqslant m+n,\\
-1,&\text{for }m+n+1\leqslant i\leqslant m+2n. 
\end{cases}
\end{align*}

The element $A=\mathop{\sum}\limits_{i,j=1}^{m+2n} a_{ij}e_{ij}\in \mathrm{End}(V)$ with respect to this basis $\{e_i\}_{1\leqslant i\leqslant m+n}$, where $e_{ij}e_k=\delta_{jk}e_i$,  belongs to $\mathfrak{osp}_{m|2n}$ if and only if it satisfies $TA=-A^{\mathrm{st}}T$, i.e.,
\begin{equation*}
a_{ij}=-(-1)^{|j|(|i|+|j|)}\varepsilon_{i}\varepsilon_{j}a_{j'i'}\quad \forall\  i,j=1,2,\cdots,m+2n.
\end{equation*}
The ortho-symplectic Lie superalgerbra $\mathfrak{osp}_{m|2n}$ can be spanned by $F_{ij}=e_{ij}-(-1)^{|j|(|i|+|j|)}\varepsilon_{i}\varepsilon_{j}e_{j'i'}$ with parity $|F_{ij}|=|i|+|j|$.  The non-degeneracy of the bilinear form $B$ ensures us to identify $V$ with $V^*$ through the $\mathfrak{osp}_{m|2n}$-module isomorphism $\Theta$ with $e_i\mapsto \varepsilon_i e_{i'}^*$.

We note that one advantage of this selection is that the diagonal matrices constitute the Cartan subalgebra of the Lie superalgebra $\mathfrak{osp}_{m|2n}$, and the upper (resp. lower) triangular matrices coincide with $\mathfrak{n}^+$ (resp. $\mathfrak{n}^-$). Another advantage is that the definition of $\mathfrak{osp}_{m|2n}$ can be rewritten as a matrix presentation, as shown in Equation \ref{eq:osppresent}.
Denote by $h_i=F_{n+i,n+i}$ and $h'_j=F_{jj}$ for all $1\leqslant i\leqslant m$ and $1\leqslant j\leqslant n$. Then $h_1,\cdots,h_m,h'_1,\cdots,h'_n$ form a basis of the Cartan subalgebra of $\mathfrak{osp}_{m|2n}$. Let 
$\{\vare_1\ldots,\vare_m,\delta_1,\delta_2,\ldots,\delta_n\}$ be 
the basis of $\mf h^\ast$ dual to  $\{h_1,\cdots,h_m,h'_1,\cdots,h'_n\}$.  
	The corresponding even and odd roots  are given by: \begin{align*}
	\Phi_\oa:=\{\pm\vare_i\pm\vare_j, \pm\delta_k \pm\delta_l, \pm 2\delta_q\},\quad\Phi_\ob:=\{\pm\vare_p\pm \delta_q\},
\end{align*} 
and 
\begin{align*}
	\Phi_\oa:=\{\pm\vare_i\pm\vare_j, \pm2\vare_p, \pm\delta_k \pm\delta_l, \pm \delta_q\},\quad\Phi_\ob:=\{\pm\vare_p\pm \delta_q, \pm\vare_p\},
\end{align*} 
where $1\leqslant i<j \leqslant m, 1\leqslant k<l \leqslant n, ~1\leqslant p \leqslant m, 1\leqslant q \leqslant n$ for $\mathfrak{osp}_{2m|2n}$ and $\mathfrak{osp}_{2m+1|2n}$, respectively.
We fix the positive system as follows:
\begin{align*}
&\Phi^+:=\{\vare_i \pm\vare_j, 2\vare_p, \delta_k\pm\delta_l \} \cup \{\vare_p \pm \delta_q\},
\end{align*}
and 
\begin{align*}
&\Phi^+:=\{\vare_i \pm\vare_j, 2\vare_p, \delta_k\pm\delta_l, \delta_q \} \cup \{\vare_p \pm \delta_q, \vare_p\},
\end{align*}
where $1\leqslant i<j \leqslant m, 1\leqslant k<l \leqslant n, ~1\leqslant p \leqslant m, 1\leqslant q \leqslant n$ for $\mathfrak{osp}_{2m|2n}$ and $\mathfrak{osp}_{2m+1|2n}$, respectively.

Define  $\Phi^-:= -\Phi^+$ as usual. The  bilinear form $\langle\_,\_\rangle:\mf h^\ast \times \mf h^\ast \rightarrow \C$, which is induced by the super-trace form, is given by $$\langle\vare_i,\vare_j\rangle =\delta_{ij},~\langle\delta_i,\delta_j\rangle =-\delta_{i,j},~ \langle\vare_i,\delta_j\rangle=0, \text{ for $1\leqslant i,j\leqslant n$}.$$

Denote by $\mathfrak{osp}_{m|2n}^{\perp}$ the subspace of $\mathrm{End}(V)$ spanned by the elements $$e_{ij}+(-1)^{|j|(|i|+|j|)}\varepsilon_i\varepsilon_j e_{j'i'} \text{ for all }i,j=1,2,\cdots, m+2n.$$ It is a complementary subspace of $\mathfrak{osp}_{m|2n}$ in $\mathrm{End}(V)$. Furthermore, $\mathrm{End}(V)=\mathfrak{osp}_{m|2n}\oplus \mathfrak{osp}_{m|2n}^{\perp}$ as $\mathfrak{osp}_{m|2n}$-modules.

\subsubsection{The periplectic Lie superalgebra $\mathfrak{p}_{n}$}\label{perp}
Let $V=V_{\bar{0}}\oplus V_{\bar{1}}$ be a vector superspace with $\mathrm{dim}V_{\bar{0}}=\mathrm{dim}V_{\bar{1}}$. Let $(\cdot,\cdot)$ be a non-degenerate odd symmetric bilinear form on $V$. The subspace of $\mathrm{End}(V)$ that preserves $(\cdot,\cdot)$ is closed under the Lie bracket and is therefore a Lie subalgebra of $\mathrm{End}(V)$. This Lie superalgebra is referred to as the periplectic Lie superalgebra, denoted by $\mathfrak{p}(V)$. If $V=\mathbb{C}^{n|n}$, $\mathfrak{p}(V)$ is also denoted by $\mathfrak{p}_{n}$. There exists a basis $B=B_0\cup B_1$ such that $B_0=\{e_1,\cdots,e_n\}$ is a basis for $V_{\bar{0}}$, $B_1=\{e_{n+1},\cdots,e_{2n}\}$ is a basis for $V_{\bar{1}}$, and
$$(e_{n+i},e_j)=(e_j,e_{n+i})=\delta_{ij},\quad (e_i,e_j)=(e_{n+i},e_{n+j})=0, $$
for all $i,j=1,2,\cdots,n$. The matrix of the bilinear form associated with the basis $B$ is given by:
\begin{align*}
\mathfrak{B}=\left( \begin{array}{cc} 0 & I_n\\
 		I_n & 0\\
 	\end{array} \right).
\end{align*}
Using $B$, we can see that $\mathfrak{p}_{n}$ can be represented as
 \begin{align*}
 	\left\{\left. \left( \begin{array}{cc} A & B\\
 		C & -A^t\\
 	\end{array} \right)\right| ~ B ~\text{is symmetric, and}~ C ~\text{is skew-symmetric}\right\}. 
 \end{align*}  
Set
\begin{align*}
i^{\prime}=\begin{cases}
n+i, &\text{ if }1\leqslant i\leqslant n,\\
i-n, &\text{ if }n+1\leqslant i\leqslant 2n.
\end{cases}
\end{align*}
For any $1\leqslant i,j\leqslant 2n$, define  $G_{ij}:=e_{ij}-(-1)^{|j|(|i|+|j|)}e_{j'i'}\in\mathfrak{p}_{n}$. Therefore, 
\begin{align}\label{GijGj'i'}
    G_{ij}=-(-1)^{|j|(|i|+|j|)}G_{j'i'}
\end{align}
for all $i,j$.
The non-degeneracy of the bilinear form $(\cdot,\cdot)$ ensures us to identify $V$ with $V^*$ by the odd $\mathfrak{p}_{n}$-module isomorphism $\Theta\colon e_i\mapsto  e_{i'}^*$. 

We fix the Cartan subalgebra  $\h= \bigoplus_{i=1}^n \C G_{ii}$ consisting of diagonal matrices. Let $\{\vare_1, \vare_2, \ldots, \vare_n\}$ denote the basis for $\h^\ast$ dual to the standard basis $\{G_{11}, G_{22},\ldots, G_{nn}\}\subset \h$. The sets of even and odd roots are given by 
 \begin{align*}
 	&\Phi_\oa = \{\vare_i-\vare_j|~1\leqslant i\neq  j \leqslant n\}, \\
 	&\Phi_\ob = \{-\vare_i-\vare_j|~1\leqslant i< j \leqslant n\}\cup \{\vare_i+\vare_j|~1\leqslant i\leqslant  j \leqslant n\}. 
 \end{align*}

We define the set $\Phi^+$ of positive roots and the set $\Phi^-$ of negative roots as follows:
\begin{align*}
&\Phi^+:=\{\vare_i-\vare_j|~1\leqslant i< j \leqslant n\}  	\cup\{\vare_i+\vare_j|~1\leqslant i\leqslant j\leqslant n\}, \\ 
&\Phi^-:=\{-\vare_i+\vare_j|~1\leqslant i< j <\leqslant n\}  	\cup\{-\vare_i-\vare_j|~1\leqslant i< j\leqslant n\}.
\end{align*}

 The corresponding set $\Phi_\oa^+$ of positive even roots and the corresponding set   $\Phi_\ob^+$ of positive odd roots are defined as follows:  \begin{align*}  	\Phi^+_\oa:=\{\vare_i-\vare_j|~1\leqslant i< j \leqslant n\}, ~\Phi_\ob^+:=\{\vare_i+\vare_j|~1\leqslant i\leqslant j\leqslant n\}.  \end{align*} 

 Finally, we define the  bilinear form $\langle\cdot,\cdot\rangle:\mf h^\ast \times \mf h^\ast \rightarrow \C$ by declaring that  $\langle\vare_i,\vare_j\rangle = \delta_{ij}$, for $1\leqslant i,j\leqslant n.$ 
 
 Denote by $\mathfrak{p}_{n}^{\perp}$ the subspace of $\mathrm{End}(V)$ spanned by the elements 
$$\left( \begin{array}{cc} A & C\\
 		B & A^t\\
 	\end{array} \right),$$
	 for all $A,B\in\mathfrak{gl}_{n} $ with $B^t=B$ and $C^t=-C$. It is a complementary subspace of $\mathfrak{p}_{n}$ in $\mathrm{End}(V)$. Moreover, $\mathrm{End}(V)=\mathfrak{p}_{n}\oplus \mathfrak{p}_{n}^{\perp}$ as adjoint $\mathfrak{p}_{n}$-modules. Denote by $\tilde{\pi}$ the natural projection associated with the direct sum decomposition. In conclusion, we have the following proposition if we set $\mathfrak{g}^{\perp}=0$ when $\mathfrak{g}=\mathfrak{gl}_{m|n}$.

\begin{proposition}\label{le:split}
Let $\mathfrak{g}$ be a classical Lie superalgebra in \eqref{eq:glist}, and let $V$ be the corresponding natural representation of $\mathfrak{g}$, then the following exact sequence is split
$$0\longrightarrow \mathfrak{g}^{\perp}\longrightarrow \mathrm{End}(V)\xlongrightarrow{\tilde{\pi}} \mathfrak{g}\longrightarrow 0.$$
Here, $\tilde{\pi}(e_{ij})$ is defined as follows:
\begin{align*}
\tilde{\pi}(e_{ij})=\begin{cases}E_{ij},&~~ \text{if} \quad\mathfrak{g}=\mathfrak{gl}_{m|n},\\
F_{ij}/2, &~~ \text{if} \quad\mathfrak{g}=\mathfrak{osp}_{m|2n},\\
G_{ij}/2,&~~ \text{if} \quad\mathfrak{g}=\mathfrak{p}_{n},\\
H_{ij}/2,&~~ \text{if} \quad\mathfrak{g}=\mathfrak{q}_{n}.
 \end{cases}
\end{align*}
\end{proposition}
\begin{remark}
The chosen split surjective homomorphism such that $\tilde{\pi}\circ \iota=\mathbf{1}$ is unique if and only if there does not exist a non-zero $\mathfrak{g}$-module homomorphism $f$ from $\mathfrak{g}^{\perp}$ to $\mathfrak{g}$. This is valid if $\mathfrak{g}$ is one of the classical Lie superalgebras in \eqref{eq:glist}.
\begin{itemize}
    \item If $\mathfrak{g}=\mathfrak{gl}_{m|n}$, it is obvious since $\mathfrak{g}=\mathrm{End}(V)$.
    \item If $\mathfrak{g}=\mathfrak{osp}_{m|2n}$, then $\mathfrak{osp}_{m|2n}$ is a simple module with the highest weight $2\varepsilon_1$ under the adjoint action. And the highest weight space is $1$-dimensional with the basis $e_{1,1'}$. Therefore, any $\mathfrak{g}$-module homomorphism $f$ from $\mathfrak{g}^{\perp}$ to $\mathfrak{g}$ is the zero homomorphism.
    \item If $\mathfrak{g}$ is a strange Lie superalgebra $\mathfrak{p}_n$ or $\mathfrak{q}_n$, then $\mathfrak{g}$ has a unique proper submodule $[\mathfrak{g},\mathfrak{g}]$ whose codimension is $1$, and $\mathfrak{g}^*$ is isomorphic  to $\mathfrak{g}^{\perp}$ by the supertrace form on $\mathrm{End}(V)$. Therefore, $\mathfrak{g}^{\perp}$ has a unique proper submodule which is a $1$-dimensional trivial module. This means any $\mathfrak{g}$-module homomorphism $f$ from $\mathfrak{g}^{\perp}$ to $\mathfrak{g}$ is the zero homomorphism.
\end{itemize}
    
\end{remark}

\subsection{Harish-Chandra homomorphism } \label{sect:harish-chandra}
Let $\mathfrak{g}$ be a semisimple Lie algebra (resp., a basic Lie superalgebra) or queer Lie superalgebra over $\mathbb{C}$ with a triangular decomposition $\mathfrak{g}=\mathfrak{n}^-\oplus\mathfrak{h}\oplus\mathfrak{n}^+$, where $\mathfrak{h}$ is a Cartan subalgebra and $\mathfrak{n}^+$ (resp., $\mathfrak{n}^-$) is the positive (resp., negative) part of $\mathfrak{g}$ corresponding to a positive root system $\Phi^+$. By using the PBW Theorem, we obtain the decomposition $\mathrm{U}(\mathfrak{g})=\mathrm{U}(\mathfrak{h})\oplus \left(\mathfrak{n}^-\mathrm{U}(\mathfrak{g})+\mathrm{U}(\mathfrak{g})\mathfrak{n}^+\right)$. Let $\zeta\colon\mathrm{U}(\mathfrak{g})\rightarrow\mathrm{U}(\mathfrak{h})=\mathrm{S}(\mathfrak{h})$ be the associated projection. The restriction  of $\zeta$ to the center $\mathcal{Z}(\mathfrak{g})$ of $\mathrm{U}(\mathfrak{g})$ is an algebra homomorphism,  and the composite $\gamma_{-\rho}\circ\zeta\colon\mathcal{Z}(\mathfrak{g})\rightarrow\mathrm{S}(\mathfrak{h})$ of  $\zeta$  with a ``shift'' by the Weyl vector $\rho$ is called the \textit{Harish-Chandra homomorphism} of $\mathrm{U}(\mathfrak{g})$. The famous Harish-Chandra isomorphism theorem states that $\gamma_{-\rho}\circ\zeta$ induces an isomorphism from $\mathcal{Z}(\mathfrak{g})$ to the algebra of $W$-invariant polynomials if $\mathfrak{g}$ is a semisimple Lie algebra, or to the algebra of $W$-invariant supersymmetric polynomials if $\mathfrak{g}$ is a basic classical Lie superalgebra, or to the algebra of $Q$-polynomials if $\mathfrak{g}$ is a queer Lie superalgebra. More details can be found in \cite[Chapter 11]{Ca05} for classical Lie algebras,  and \cite[Section 2.2]{CW12}, \cite[Chapter 13]{Mu12}, \cite{Se83} for classical Lie superalgebras.

Let $\mathcal{X}_m=(x_1,\ldots, x_m),\ \mathcal{Y}_m=(y_1,\ldots, y_m)$ be two sets of indeterminates. A polynomial $f\in\mathbb{C}[\mathcal{X}_m,\mathcal{Y}_n]$ is called {\textit supersymmetric} if it satisfies the following conditions:
\begin{itemize}
\item $f$ is symmetric in $x_1,\ldots, x_m$;
\item $f$ is symmetric in $y_1,\ldots, y_n$;
\item The polynomial obtained from $f$ by setting $x_m=y_n=t$ is independent of $t$.
\end{itemize} 

For example, if $r\geqslant 1$, the  polynomial
$$p_{m,n}^{(r)}=(x_1^r+\ldots+x_m^r)+(-1)^{r-1}(y_1^r+\ldots+y_n^r)$$
is a supersymmetric polynomial.

Denote $I(\mathcal{X}_m,\mathcal{Y}_n)$ as the set of all supersymmetric polynomials in $x_1,\cdots,x_m$ and $y_1,\cdots,y_n$.
A polynomial $f\in\mathbb{C}[\mathcal{X}_n]$ is called \textit{$Q$-polynomial ring} if it satisfies the following conditions:
\begin{itemize}
\item $f$ is symmetric in $x_1,\ldots, x_n$.
\item The polynomial obtained from $f$ by setting $x_i=-x_j=t$ with $i\neq j$ is independent of $t$.
\end{itemize} 

For example,  the  polynomial
$$p_{n}^{(2r-1)}=x_1^{2r-1}+\ldots+x_n^{2r-1}$$
is a $Q$-polynomial.
Denote $Q(\mathcal{X}_n)$ as the set of all $Q$-polynomials in $x_1,\ldots,x_n$.

\begin{theorem}\cite[Theorem 12.4.1]{Mu12}
The supersymmetric polynomial $I(\mathcal{X}_m; \mathcal{Y}_n)$  is generated by the polynomials $p_{m,n}^{(r)}$ for $r\in\mathbb{N}$.
\end{theorem}

\begin{theorem}\cite[Section 8, Chapter 3 ]{Mac}
The $Q$-polynomial $Q(x_1,\ldots,x_n)$ is generated by the polynomials $p_{n}^{(2r-1)}$ for $r\in\mathbb{N}$.
\end{theorem}

Recall that the Weyl group $W$ for $\mathfrak{gl}_{m|n}$ is $W=\mathfrak{S}_m\times \mathfrak{S}_n$ and $S(\mathfrak{h})$ can be identified with $\mathbb{C}[\mathfrak{h}]:=\mathbb{C}[h_1,\ldots,h_m;h'_1,\ldots, h'_n]$. The Weyl group $W$ for $\mathfrak{osp}_{2m+1|2n}$ is $W=(\mathfrak{S}_m\ltimes\mathbb{Z}_2^m)\times (\mathfrak{S}_n\ltimes\mathbb{Z}_2^n)$ and the Weyl group $W'$ of  $\mathfrak{osp}_{2m|2n}$ an index 2 subgroup of $W=(\mathfrak{S}_m\ltimes\mathbb{Z}_2^m)\times (\mathfrak{S}_n\ltimes\mathbb{Z}_2^n)$ with only an even number of signs in $\mathbb{Z}_2^m$ permitted. Let $S(\mathfrak{h})^{W}_{sup}$ be the subalgebra of $S(\mathfrak{h})$ consisting of all $W$-invariant supersymmetric functions (resp. $Q$-polynomial if $\mathfrak{g}=\mathfrak{q}_n$). Set
\begin{align*}
    &I(\mathfrak{h})=I(h_1,\ldots,h_m; h'_1,\ldots,h'_n),\\
    &J(\mathfrak{h})=I(h_1^2,\ldots,h_m^2; (h'_1)^2,\ldots,(h'_n)^2),\\
    &Q(\mathfrak{h})=Q(h_1,\cdots,h_n).
\end{align*}

\begin{theorem}[Gorelik-Kac-Sergeev]
The map $\gamma_{-\rho}\circ\zeta$ is an algebra isomorphism from $\mathcal{Z}(\mathfrak{g})$ to $S(\mathfrak{h})^{W}_{sup}$. Moreover, 
\begin{align*}
S(\mathfrak{h})^{W}_{sup}=
\begin{cases}
I(\mathfrak{h}),&\text{if }\mathfrak{g}=\mathfrak{gl}_{m|n},\\
J(\mathfrak{h}),&\text{if }\mathfrak{g}=\mathfrak{osp}_{2m+1|2n},\\
J(\mathfrak{h})+\Phi_{m,n}\mathbb{C}[\mathfrak{h}]^{W},&\text{if }\mathfrak{g}=\mathfrak{osp}_{2m|2n},\\
Q(\mathfrak{h}), &\text{if }\mathfrak{g}=\mathfrak{q}_{n},
\end{cases}
\end{align*}
where $\Phi_{m,n}=(h_1\cdots h_m)\mathop{\prod}\limits_{i,j}(h_i^2-(h'_j)^2)$ and $W=(\mathfrak{S}_m\ltimes\mathbb{Z}_2^m)\times (\mathfrak{S}_n\ltimes\mathbb{Z}_2^n)$.
\end{theorem}

\subsection{Brauer-Schur-Weyl-Sergeev duality}\label{sect:schurweyl}
The Schur-Weyl duality is a signficantly influential topic in representation theory as it enables the construction of related associative algebras. In this subsection, we formulate the Schur-Weyl dualities for classical Lie superalgebras. Subsequently, we will utilize these dualities to construct the Gelfand invariant in the  subsequent subsection. 
 
Let $\mathfrak{g}$ be a classical Lie superalgebra as in \eqref{eq:glist} and let $V$ be the natural representation of the $\mathfrak{g}$-module. Consequently, the $k$-fold tensor product $V^{\otimes k}$ is naturally a $\mathfrak{g}$-module, denoted  by $\Phi_k$. On the other hand, the action $\Psi_k$ of the symmetric group $\mathfrak{S}_k$ on $V^{\otimes k}$ is defined by
\begin{align}
&(i,i+1)\cdot v_1\otimes v_2\otimes\cdots\otimes v_i\otimes v_{i+1}\otimes\cdots \otimes v_k\nonumber\\
=&(-1)^{|v_i||v_{i+1}|}v_1\otimes v_2\otimes\cdots\otimes v_{i+1}\otimes v_{i}\otimes\cdots \otimes v_k,\quad1\leqslant i\leqslant k-1,
\end{align}
where $(i,j)$ denotes a transposition in $\mathfrak{S}_k$ and $v_i, v_{i+1}$ are homogeneous elements. 
The Clifford superalgebra $\mathcal{C}_k$ , as defined in \cite[Definition 3.33]{CW12}, is the $\mathbb{C}$-superalgebra generated by the odd elements $c_1,c_2,\cdots,c_k$, subject to the relations
$$c_i^2=1, \quad  c_ic_j=-c_jc_i,\quad 1\leqslant i\neq j\leqslant k.$$
The symmetric group $\mathfrak{S}_k$ acts as automorphisms on the algebra $\mathcal{C}_k$ naturally. We will refer to the semi-direct product $\mathcal{H}_k\colon=\mathbb{C}\mathfrak{S}_k\ltimes\mathcal{C}_k $ as the \textit{Hecke-Clifford algebra}, where  $\sigma c_i=c_{\sigma(i)}\sigma$ for all $\sigma\in\mathfrak{S}_k$. Note that the algebra $\mathcal{H}_k$ is naturally a superalgebra by letting each $\sigma\in \mathfrak{S}_k$ be even and each $c_i$ be odd.
The tensor space $V^{\otimes k}$ is a representation of $\mathfrak{gl}_{n|n}$, and hence of its subalgebra $\mathfrak{q}_{n}$. Moreover, $V^{\otimes k} $ is also a representation of the symmetric group $\mathfrak{S}_k$. Define the Clifford superalgebra $\mathcal{C}_k$  action on $V^{\otimes k}$ by
$$c_i\cdot (v_1\otimes v_2\otimes\cdots \otimes v_k)=(-1)^{|v_1|+|v_2|+\cdots+|v_{i-1}|}v_1\otimes \cdots\otimes v_{i-1}\otimes \mathcal{P}v_i\otimes \cdots\otimes v_k,$$
where $\mathcal{P}$ is defined in \eqref{eq:cliffact}.
\begin{definition}
Let $k\in\mathbb{Z}_+$ and $\delta\in\mathbb{C}$. The \textit{Brauer algebra} $\mathcal{B}_k(\delta)$ is an associative unital  $\mathbb{C}$-algebra generated by the elements 
$s_i$ and $e_i$, for $i=1,\ldots,k-1,$ subject to the following relations
\begin{align*} 
&s_i^2=1, \ e_i^2=\delta e_i,\ e_is_i=e_i=s_ie_i, &&1\leqslant i\leqslant k-1,\\
&s_is_j=s_js_i,\ s_ie_j=e_js_i,\ e_ie_j=e_je_i,&&2\leqslant|i-j|,\\
&s_is_{i+1}s_i=s_{i+1}s_is_{i+1},\ e_ie_{i+1}e_i=e_{i+1},\ e_{i+1}e_{i}e_{i+1}=e_i, &&1\leqslant i\leqslant k-2,\\
&s_ie_{i+1}e_i=s_{i+1}e_i,\ s_{i+1}e_ie_{i+1}=s_ie_{i+1}, &&1\leqslant i\leqslant k-2.
\end{align*}
\end{definition}
Define the contraction map $c\in \mathrm{End}(V^{\otimes 2})$ as follows:
$$c(v_1\otimes v_2)=( v_1,v_2)\mathop{\sum}\limits_{i=1}^{n}\varepsilon_ie_i\otimes e_{i'} .$$

It is straightforward to show that
$$x\cdot c(v_1\otimes v_2)=c( x\cdot(v_1\otimes v_2))=0,\text{ for any }v_1, v_2\in V,\text{ and }  x\in\mathfrak{osp}_{m|2n}.$$ Therefore, $c\in \mathrm{End}_{\mathfrak{osp}_{m|2n}}(V^{\otimes 2})$ and $c(V^{\otimes 2})$ is an 1-dimensional submodule of $V\otimes V$. Define $e_i\in \mathrm{End}_{\mathfrak{osp}_{m|2n}}(V^{\otimes k})$ for $k\geqslant2$ and $i=1,2,\cdots,k-1$ by
$$e_i=I_V^{\otimes (i-1)}\otimes c\otimes I_V^{\otimes (k-i-1)}.$$

\begin{definition}
The \textit{Periplectic Brauer algebra} $\mathcal{B}^-_k(0), k\in\mathbb{Z}_+$, is an associative unital  $\mathbb{C}$-algebra generated by 
$$s_i,\ e_i,\quad\text{for } i=1,\ldots,k-1,$$
 subject to the relations
\begin{align*} 
&s_i^2=1,\ e_i^2=0,\ e_is_i=e_i,\ s_ie_i=-e_i, &&1\leqslant i\leqslant k-1,\\
&s_is_j=s_js_i,\ s_ie_j=e_js_i,\ e_ie_j=e_je_i,&&2\leqslant|i-j|,\\
&s_is_{i+1}s_i=s_{i+1}s_is_{i+1},\ e_ie_{i+1}e_i=-e_{i+1},\ e_{i+1}e_{i}e_{i+1}=-e_i, &&1\leqslant i\leqslant k-2,\\
&e_ie_{i+1}s_i=-e_is_{i+1},\ s_{i+1}e_ie_{i+1}=-s_ie_{i+1}, &&1\leqslant i\leqslant k-2.
\end{align*}
\end{definition}

\begin{remark}
The Periplectic Brauer algebra mentioned here is the opposite algebra of the one described in \cite{Mo03}, as the action on $V^{\otimes k}$ occurs on different sides.
\end{remark}
In this case, the contraction map $c\in \mathrm{End}(V^{\otimes 2})$ is defined as
$$c(v_1\otimes v_2)=( v_1,v_2)\mathop{\sum}\limits_{i=1}^{n}(e_i\otimes e_{n+i}-e_{n+i}\otimes e_i).$$

Since the $\mathfrak{p}_{n}$-module invariant is annihilated by $\mathfrak{p}_{n}$, it can be easily shown that $x\cdot c(v_1\otimes v_2)=0=c( x\cdot(v_1\otimes v_2))$,  for any $v_1, v_2\in V,\text{ and }  x\in\mathfrak{p}_{n}$. Therefore, $c\in \mathrm{End}_{\mathfrak{p}_{n}}(V^{\otimes 2})$ and $c(V^{\otimes 2})$ is a 1-dimensional submodule of $V\otimes V$. Additionally, $e_i\in \mathrm{End}_{\mathfrak{p}_{n}}(V^{\otimes k})$ is defined for $k\geqslant2$ and $i=1,2,\cdots,k-1$ as
$$e_i=I_V^{\otimes (i-1)}\otimes c\otimes I_V^{\otimes (k-i-1)}.$$

Let $\mathcal{A}_k$ denote the symmetric group algebra $\mathbb{C}\mathfrak{S}_k$, the Hecke-Clifford algebra $\mathcal{H}_k$, the Brauer algebra $\mathcal{B}_k(2m+1-2n)$, or the Periplectic Brauer algebra $\mathcal{B}^-_k(0)$ if $\mathfrak{g}$ is $\mathfrak{gl}_{m|n},\ \mathfrak{q}_{n},\ \mathfrak{osp}_{2m+1|2n}$ or $\mathfrak{p}_{n}$, respectively. Then, the following theorem, known as Schur-Sergeev duality, holds:
\begin{theorem}\label{thm:sw}[Schur-Sergeev duality]
Let $\mathfrak{g}$ be the Lie superalgebra listed in \eqref{eq:glist}, and let $V$ be the corresponding natural representation. Then, the actions of $\Phi_k$ and $\Phi_k$ on $V^{\otimes k}$ commute with each other. Furthermore, the algebra homomorphism
\begin{align*}
\mathcal{A}_k\xlongrightarrow{\Psi_k}\mathrm{End}_{\mathfrak{g}}(V^{\otimes k})
\end{align*}
is surjective, except in the case when $\mathfrak{g}=\mathfrak{osp}_{2m|2n}$. 
\end{theorem}
The surjectivity of $\Psi_k$ is of significant importance in this paper, as it guarantees that one can find all the central elements of the corresponding universal enveloping algebra. 


\section{Invariants for classical Lie superalgebras}\label{se:Centers}
In this section, we will investigate the relationship between the tensor algebra $T(\mathfrak{g})$, the supersymmetric algebra $S(\mathfrak{g})$ and the universal enveloping algebra $\mathrm{U}(\mathfrak{g})$ of a Lie superalgebra $\mathfrak{g}$, see the non-commutative diagram \eqref{eq:TSUrelation}. Consequently, we obtain the precise relationship between their $\mathfrak{g}$-invariant subalgebras, see the Theorem \ref{etaeta'}. As an application, we employ the formula and the super-analog of the Schur-Weyl dualities to determine the Gelfand invariants and the generators of the central elements.
We refer the reader to \cite{M18} for a fairly comprehensive introduction to Gelfand invariants of Lie superalgebras.

For any two sequences $\mathbf{x}=(x_1,x_2,\cdots,x_k),\mathbf{y}=(y_1,y_2,\cdots,y_k)\in\mathbb{Z}_2^k$, we denote 
\begin{align*}
p(\mathbf{x},\mathbf{y})\colon=\mathop{\prod}\limits_{i>j}(-1)^{x_iy_j}.
\end{align*}
It is clear that
$p(\mathbf{x}+\mathbf{y},\mathbf{z})=p(\mathbf{x},\mathbf{z})p(\mathbf{y},\mathbf{z}), \ p(\mathbf{x},\mathbf{y}+\mathbf{z})=p(\mathbf{x},\mathbf{y})p(\mathbf{x},\mathbf{z})$
and
$$p(\mathbf{x},\mathbf{y})p(\mathbf{y},\mathbf{x})=\mathop{\prod}\limits_{i=1}^{k}(-1)^{x_iy_i}\mathop{\prod}\limits_{i,j=1}^{k}(-1)^{x_iy_j},$$
for all $\mathbf{x},\mathbf{y},\mathbf{z}\in\mathbb{Z}_2^k$.

For any $\mathbf{x}=(x_1,x_2,\cdots,x_k)\in\mathbb{Z}_2^k$ and $\sigma\in \mathfrak{S}_k$, we define $\gamma(\mathbf{x},\sigma)$ by the rule
\begin{align*}
\gamma(\mathbf{x},\sigma)\colon=\mathop{\prod}\limits_{\substack{i<j\\ \sigma(i)>\sigma(j)}}(-1)^{x_{\sigma(i)}x_{\sigma(j)}}.
\end{align*}

Let $V$ be a $\mathbb{Z}_2$-graded vector space, and let $T(V)$ be the tensor algebra on $V$. The \textit{supersymmetric algebra} $S(V)$ on $V$ is defined as the quotient of $T(V)$ by the ideal generated by 
$v\otimes w-(-1)^{|v||w|}w\otimes v$
for all homogeneous elements $v,w\in V$. Therefore, we have
$$v_{1}v_{2}\cdots v_{k}=\gamma(\mathbf{v},\sigma)v_{\sigma(1)}v_{\sigma(2)}\cdots v_{\sigma(k)},$$
for all $\mathbf{v}=v_1 v_2\cdots v_k\in S(V)$ and $\sigma\in\mathfrak{S}_k$, where $v_1,v_2,\cdots,v_k\in V$. 

Here, for convenience, we utilize the same notation $\mathbf{v}$ in different contexts, which will not lead to any confusion. In the case of $p(\mathbf{v},\mathbf{w})$ and $\gamma(\mathbf{v},\sigma)$, the symbol $\mathbf{v}$ represents the parities $(|v_1|,|v_2|,\cdots,|v_k|)$, and later we also employ $\mathbf{v}$ to denote a simple tensor $v_1\otimes v_2\otimes \cdots \otimes v_k\in V^{\otimes k}$. 

Define the action of $\mathfrak{S}_k$ on $V^{\otimes k}$ to be
\begin{align}\label{sigma}
\sigma\cdot \mathbf{v}\colon=\gamma(\mathbf{v},\sigma^{-1})(v_{\sigma^{-1}(1)}\otimes v_{\sigma^{-1}(2)}\otimes \cdots\otimes v_{\sigma^{-1}(k)}),
\end{align}
for all $\sigma\in \mathfrak{S}_k$. Let  $\mathbf{v}_{\sigma}=\sigma^{-1}\cdot \mathbf{v}$, then we have 
\begin{align}\label{gamma}
\gamma(\mathbf{v},\sigma\tau)=\gamma( \mathbf{v}_{\sigma},\tau)\gamma(\mathbf{v},\sigma),
\end{align}
for all $\mathbf{v}\in V^{\otimes k}$ and $\sigma,\tau\in\mathfrak{S}_k$.

\begin{lemma}\label{eq::gammaprelation}
The functions $\gamma(\cdot,\cdot)$ and $p(\cdot,\cdot)$ satisfy the following equation: \begin{align}\label{eq::gammap}
\gamma(\mathbf{u}+\mathbf{v},\sigma)=p(\mathbf{u},\mathbf{v})p(\mathbf{u}_{\sigma},\mathbf{v}_{\sigma})\gamma(\mathbf{u},\sigma)\gamma(\mathbf{v},\sigma),
\end{align}
for all simple tensors $\mathbf{u}=u_1\otimes u_2\otimes \cdots \otimes u_k,~\mathbf{v}=v_1\otimes v_2\otimes \cdots \otimes v_k\in V^{\otimes k}$ and $\sigma\in\mathfrak{S}_k$.
\end{lemma}
\begin{proof}
Let $A$ be a free commutative superalgebra generated by homogeneous elements $a_1,\cdots,a_k$ and $b_1,\cdots,b_k$ with parities $|a_i|=|u_i|$ and $|b_j|=|v_j|$ for $i,j=1,\cdots,k$. Then, we have the following equations: 
\begin{align}\label{asigmabsigma}
 a_{\sigma(1)}\cdots a_{\sigma(k)}=&\gamma(\mathbf{u},\sigma) a_1a_2\cdots a_k,\\
b_{\sigma(1)}\cdots b_{\sigma(k)}=&\gamma(\mathbf{v},\sigma) b_1b_2\cdots b_k,\\
 a_{\sigma(1)}b_{\sigma(1)}\cdots a_{\sigma(k)}b_{\sigma(k)}=&\gamma(\mathbf{u}+\mathbf{v},\sigma) a_1b_1a_2b_2\cdots a_kb_k,
\end{align}
for any $\sigma\in \mathfrak{S}_k$.
On the other hand,
\begin{align}
a_1b_1a_2b_2\cdots a_kb_k&=p(\mathbf{u},\mathbf{v})a_1a_2\cdots a_k b_1b_2\cdots b_k,\\
a_{\sigma(1)}b_{\sigma(1)}\cdots a_{\sigma(k)}b_{\sigma(k)}&=p(\mathbf{u}_{\sigma},\mathbf{v}_{\sigma})a_{\sigma(1)}\cdots a_{\sigma(k)}b_{\sigma(1)}\cdots b_{\sigma(k)}.\label{absigma}
\end{align}
The lemma follows from \eqref{asigmabsigma} to \eqref{absigma}.
\end{proof}
In particular,
\begin{align}\label{psigma}
p(\mathbf{v},\mathbf{v})=p(\mathbf{v}_{\sigma},\mathbf{v}_{\sigma}),
\end{align}
for all $\mathbf{v}\in V^{\otimes k}$ and $\sigma\in\mathfrak{S}_k$ if we set $\mathbf{u}=\mathbf{v}$ in \eqref{eq::gammap}. 

It is obvious that $S(\mathfrak{g})$ is a $\mathbb{Z}_+$-graded algebra, meaning that, $S(\mathfrak{g})=\bigoplus\limits_{k=0}^{\infty}S^k(\mathfrak{g})$, where $S^k(\mathfrak{g})$ is the image of $\mathfrak{g}^{\otimes k }$ under the natural map $\eta\colon T(\mathfrak{g})\to S(\mathfrak{g})$. Moreover, $S(\mathfrak{g})$ is also a $\mathfrak{g}$-module by the adjoint action. In other words, for all elements  $a,x_1,x_2,\cdots,x_n\in\mathfrak{g}$, the action of $a$ on $x_1x_2\cdots x_n$ is defined as 
$$a\cdot x_1x_2\cdots x_n\colon =\mathop{\sum}\limits_{i=1}^n (-1)^{|a|(|x_1|+\cdots+|x_{i-1}|)} x_1\cdots x_{i-1}[a,x_i]x_{i+1}\cdots x_n.$$
Set $I_k=I\cap \mathfrak{g}^{\otimes k}$, where $I$ is the ideal of $T(\mathfrak{g})$ generated by $x\otimes y-(-1)^{|x||y|}y\otimes x$ for all $x,y \in\mathfrak{g}$. Let $\eta_k$ be the restriction of $\eta$ to $\mathfrak{g}^{\otimes k}$. Define a linear map $\omega_k:S^k(\mathfrak{g})\to \mathfrak{g}^{\otimes k}$ by
$$\omega_k(x_1x_2\cdots x_k)=\frac{1}{k!}\mathop{\sum}\limits_{\sigma\in\mathfrak{S}_k}\gamma(\mathbf{x},\sigma)x_{\sigma(1)}\otimes x_{\sigma(2)}\otimes\cdots\otimes x_{\sigma(k)},$$
for all homogeneous elements $x_1,x_2,\cdots,x_k\in \mathfrak{g}$.  It can be verified that $\eta_k\circ\omega_k$ is the identity on $S^k(\mathfrak{g})$, thus $S^k(\mathfrak{g})$ is a direct summand of $\mathfrak{g}^{\otimes k}$ and the following short exact sequence is split:
\begin{align}\label{split}
0\to I\to T(\mathfrak{g})\xrightarrow{\eta} S(\mathfrak{g}) \to 0.
\end{align}


Consider the following non-commutative diagram: 
\begin{align}\label{eq:TSUrelation}
\xymatrix{
T(\mathfrak{g})\ar[r]^{\eta}\ar[dr]_{\eta'}&S(\mathfrak{g})\ar@{->}[d]^{\psi} \\
&\mathrm{U}(\mathfrak{g}),
}
\end{align}
where $\eta'$ is the canonical map and
\begin{align}\label{psi}
\psi(x_1x_2\cdots x_n)=\frac{1}{n!}\mathop{\sum}\limits_{\sigma\in\mathfrak{S}_n}\gamma(\mathbf{x},\sigma)x_{\sigma(1)} x_{\sigma(2)}\cdots x_{\sigma(n)},
\end{align}
for all homogeneous elements $x_1,x_2,\cdots,x_n\in\mathfrak{g}$.
Note that these maps are $\mathrm{U}(\mathfrak{g})$-module homomorphisms, hence they map $\mathfrak{g}$-invariants to $\mathfrak{g}$-invariants. Since $\eta$ is split and $\psi$ is an isomorphism, the restrictions of $\eta$ on $T(\mathfrak{g})^{\mathfrak{g}}$ and $\psi$ on $S(\mathfrak{g})^{\mathfrak{g}}$ are both surjective. Therefore, a natural question arises:

{\bf Question:} Is the restriction of $\eta'$ on the $\mathfrak{g}$-invariants $T(\mathfrak{g})^{ \mathfrak{g} }$  surjective?
 
 If $\mathfrak{g}$ is a reductive Lie algebra over a field with characteristic zero, the restriction $\eta'$ is indeed surjective. The rest of this paper is devoted to providing a positive answer to this question when $\mathfrak{g}$ is a classical Lie superalgebra in the list \eqref{eq:glist}.

 Let $\{e_i\}$ be the standard homogeneous basis of $V$ and let $\{e_i^*\}$ be its dual basis for $V^*$. For a sequence $\mathbf{I}=(i_1,i_2,\cdots,i_k)$ of length $k$, or simply a $k$\textit{-sequence}, with $1\leqslant i_s\leqslant \mathrm{dim}V$ for all $s=1,2,\cdots,k$, we denote 
$$e_\mathbf{I}=e_{i_1}\otimes \cdots\otimes e_{i_k}\in V^{\otimes k}.$$
Thus, $e_{\mathbf{I}}$ with all $k$-sequences $\mathbf{I}$ forms a basis of $V^{\otimes k}$. 
For any two $k$-sequences $\mathbf{I}$ and  $\mathbf{J}$, let  $p(\mathbf{I},\mathbf{J})=p(|e_{\mathbf{I}}|,|e_{\mathbf{J}}|)$ and $p(\mathbf{I}+\mathbf{J},\mathbf{I})=p(\mathbf{I},\mathbf{I})p(\mathbf{J},\mathbf{I})$.
Define the map
$\Omega_k:\mathrm{End}(V^{\otimes k})\longrightarrow \mathrm{End}(V)^{\otimes k}$
by
\begin{align}\label{Omega}
\Omega_k(f)=\sum_{\mathbf{I},\mathbf{J}}p(\mathbf{I}+\mathbf{J},\mathbf{I})a_{i_1\cdots i_k}^{j_1\cdots j_k}e_{j_1i_1}\otimes\cdots\otimes e_{j_ki_k}.
\end{align} 
Here $e_{ij}\in \mathrm{End}(V)$ is a matrix unit, i.e., $e_{ij}(e_k)=\delta_{jk}e_i$,  $f(e_{\mathbf{I}})=\mathop{\sum}\limits_{\mathbf{J}}a_{i_1\cdots i_k}^{j_1\cdots j_k}e_{\mathbf{J}}$ and $a_{i_1\cdots i_k}^{j_1\cdots j_k}\in\mathbb{C}$. 
\begin{remark}
The map $\Omega_k$ is the composition of the following natural maps: 
$$\mathrm{End}(V^{\otimes k})\longrightarrow V^{\otimes k}\otimes (V^{\otimes k})^*\longrightarrow V^{\otimes k}\otimes (V^*)^{\otimes k}\longrightarrow (V\otimes V^*)^{\otimes k}\longrightarrow \mathrm{End}(V)^{\otimes k}.$$
\end{remark}

Note that $\mathrm{End}(V^{\otimes k})$ and $\mathrm{End}(V)^{\otimes k}$ are both $\mathrm{U}(\mathfrak{g})$-module superalgebras, and $\Omega$ is an isomorphism of $\mathrm{U}(\mathfrak{g})$-module superalgebras. Therefore, the  restriction of $\Omega_k$ to $\mathrm{End}_{\mathfrak{g}}(V^{\otimes k})$  is a superalgebra isomorphism from $\mathrm{End}_{\mathfrak{g}}(V^{\otimes k})$ to $\left[\mathrm{End}(V)^{\otimes k}\right]^{\mathfrak{g}}$.

Let $\pi$ be the restriction of $\bigoplus\limits_{k\geqslant0}\tilde{\pi}^{\otimes k}$ on $\bigoplus\limits_{k\geqslant0}\left[\mathrm{End}(V)^{\otimes k}\right]^{\mathfrak{g}}$. According to Lemma \ref{le:split}, we have the following diagram:
 $$
\xymatrix{\bigoplus\limits_{k\geqslant0}\mathcal{A}_k\ar@{>}[r]^(.37){\Psi} &\bigoplus\limits_{k\geqslant0}\mathrm{End}_{\mathfrak{g}}(V^{\otimes k})\ar@{->>}[r]^(.46){\Omega}&\bigoplus\limits_{k\geqslant0}\left[\mathrm{End}(V)^{\otimes k}\right]^{\mathfrak{g}}\ar@{->>}[r]^(.62){\pi}&
T(\mathfrak{g})^{\mathfrak{g}}\ar@{->>}[r]^{\eta}\ar[dr]^{\eta'}&S(\mathfrak{g})^{\mathfrak{g}}\ar@{->>}[d]^{\psi}\\
&&&&\mathcal{Z}(\mathfrak{g}),
}
$$
where $\Psi=\bigoplus\limits_{k\geqslant0}\Psi_k$,  $\Omega=\bigoplus\limits_{k\geqslant0}\Omega_k$,   and
\begin{equation}
\mathcal{A}_k=\begin{cases}
\mathbb{C}\mathfrak{S}_k  &\text{ if } \mathfrak{g}=\mathfrak{gl}_{m|n},\\
\mathcal{H}_k &\text{ if }\mathfrak{g}=\mathfrak{q}_{n},\\
\mathcal{B}_k(m-2n)  &\text{ if } \mathfrak{g}=\mathfrak{osp}_{m|2n},\\
\mathcal{B}^-_k(0)  &\text{ if } \mathfrak{g}=\mathfrak{p}_{n}.
\end{cases}
\end{equation}

It is important to emphasize that the map $\Psi$ is a surjective superalgebra homomorphism, except in the case of $\mathfrak{osp}_{2m|2n}$, as demonstrated by the super analogue of Schur-Weyl duality (see, Theorem \ref{thm:sw}). Additionally, $\Omega$ is an isomorphism of $\mathrm{U}(\mathfrak{g})$-module superalgebras. In the subsequent subsection, we will proceed with the calculation $\mathfrak{z}_{\sigma^{-1}}$ for any $k$-cycle $\sigma=(k,k-1\cdots 1)$.

Recall that the action of $\sigma^{-1}$ on $V^{\otimes k}$ in (\ref{sigma}) and the isomorphism $\Omega$ defined in \eqref{Omega}. Therefore, the invariant corresponding to $\sigma^{-1}$ is
\begin{align}\label{sign}
\theta_{\sigma^{-1}}=\mathop{\sum}\limits_{\mathbf{I}}p( \mathbf{I}+\mathbf{I}_{\sigma},\mathbf{I})\gamma(\mathbf{I},\sigma)(e_{i_{\sigma(1)}i_1}\otimes \cdots\otimes e_{i_{\sigma(k)}i_k}),
\end{align}
where  $\gamma(\mathbf{I},\sigma)=\gamma(e_{\mathbf{I}},\sigma)$ and $\mathbf{I}_{\sigma}=\sigma^{-1}\cdot e_{\mathbf{I}}$.  This is a $\mathfrak{g}$-invariant element of $\mathrm{End}(V)^{\otimes k}$ associated with $\sigma^{-1}$ for any $\sigma\in\mathfrak{S}_k$. In particular,
if $\mathfrak{g}=\mathfrak{gl}_{m|n}$, then these elements $\theta_{\sigma}$ with all permutations $\sigma$ span the superspace $\left[\mathrm{End}(V)^{\otimes k}\right]^{\mathfrak{g}}$. 

Denote by $\mathfrak{z}_{a}=\eta'\circ \pi \circ \Omega\circ \Psi(a)$, for all $a\in \bigoplus\limits_{k\geqslant0}\mathcal{A}_k$. 

\begin{example}\label{examplegelfand}
Suppose $\sigma=(12\cdots k)$ is a cycle, then \begin{align*}
&p(\mathbf{I},\mathbf{I})=\mathop{\prod}\limits_{j>l}(-1)^{|e_{i_j}||e_{i_l}|},\quad
p(\mathbf{I}_{\sigma^{-1}},\mathbf{I})=\mathop{\prod}\limits_{j\geqslant l,j\neq k}(-1)^{|e_{i_j}||e_{i_l}|},\text{ and }\\
&\gamma(\mathbf{I},\sigma^{-1})=(-1)^{|e_{i_k}|(|e_{i_1}|+\cdots+|e_{i_{k-1}}|)}.
\end{align*}
Hence, $p(\mathbf{I}+\mathbf{I}_{\sigma^{-1}},\mathbf{I})\gamma(\mathbf{I},\sigma^{-1})=(-1)^{|e_{i_1}|+\cdots+|e_{i_{k-1}}|}$ and
\begin{align*}
\theta_{\sigma}=\mathop{\sum}\limits_{\mathbf{I}}(-1)^{|e_{i_1}|+\cdots+|e_{i_{k-1}}|}(e_{i_{k}i_1}\otimes e_{i_1i_2}\otimes\cdots\otimes e_{i_{k-1}i_k}).
\end{align*}
In particular, if $\mathfrak{g}=\mathfrak{gl}_{m|n}$ is the general linear Lie superalgebra, then $$\mathfrak{z}_{\sigma}=\mathop{\sum}\limits_{\mathbf{I}}(-1)^{|e_{i_1}|+\cdots+|e_{i_{k-1}}|}E_{i_{k}i_1}E_{i_1i_2}\cdots E_{i_{k-1}i_k}$$ is the Gelfand invariant of $\mathfrak{gl}_{m|n}$.
\end{example}

\begin{theorem}\label{etaeta'}
Let $\mathfrak{g}$ be a classical Lie superalgebra as defined in \eqref{eq:glist} and let $\sigma\in\mathfrak{S}_k$. Then,
\begin{align*}
\psi\circ\eta\circ \pi(\theta_{\sigma})=\frac{1}{k!}\mathop{\sum}\limits_{\tau\in \mathfrak{S}_k}\eta'\circ \pi(\theta_{\tau^{-1}\sigma\tau})=\frac{1}{k!}\mathop{\sum}\limits_{\tau\in \mathfrak{S}_k}\mathfrak{z}_{\tau^{-1}\sigma\tau}.
\end{align*}
\end{theorem}
\begin{proof}
By direct calculation, we have
\begin{align*}
&\psi\circ\eta\circ \pi(\theta_{\sigma^{-1}})\\
=&\mathop{\sum}\limits_{\mathbf{I}}p(\mathbf{I}+\mathbf{I}_{\sigma},\mathbf{I})\gamma(\mathbf{I},\sigma)\frac{1}{k!}\mathop{\sum}\limits_{\tau\in \mathfrak{S}_k}\gamma(\mathbf{I}_{\sigma}+\mathbf{I},\tau)(\tilde{\pi}(e_{i_{\sigma(\tau(1))}i_{\tau(1)}}) \cdots \tilde{\pi}(e_{i_{\sigma(\tau(k))}i_{\tau(k)}}))\\
=&\frac{1}{k!}\mathop{\sum}\limits_{\tau\in \mathfrak{S}_k}\mathop{\sum}\limits_{\mathbf{I}}p(\mathbf{I}, \mathbf{I})p(\mathbf{I}_{\sigma\tau},\mathbf{I}_{\tau})
\gamma(\mathbf{I},\sigma)\gamma(\mathbf{I},\tau)\gamma(\mathbf{I}_{\sigma},\tau)\tilde{\pi}(e_{i_{\tau(\tau^{-1}\sigma\tau(1))}i_{\tau(1)}}) \cdots \tilde{\pi}(e_{i_{\tau(\tau^{-1}\sigma\tau(k))}i_{\tau(k)}})\\
=&\frac{1}{k!}\mathop{\sum}\limits_{\tau\in \mathfrak{S}_k}\mathop{\sum}\limits_{\mathbf{I}}p(\mathbf{I}_{\tau},\mathbf{I}_{\tau})p(\mathbf{I}_{\tau(\tau^{-1}\sigma\tau)}),\mathbf{I}_{\tau})\gamma(\mathbf{I}_{\tau},\tau^{-1}\sigma\tau)
\tilde{\pi}(e_{i_{\tau(\tau^{-1}\sigma\tau(1))}i_{\tau(1)}}) \cdots \tilde{\pi}(e_{i_{\tau(\tau^{-1}\sigma\tau(k))}i_{\tau(k)}})\\
=&\frac{1}{k!}\mathop{\sum}\limits_{\tau\in \mathfrak{S}_k}\mathop{\sum}\limits_{\mathbf{I}}p(\mathbf{I},\mathbf{I})p(\mathbf{I}_{\tau^{-1}\sigma\tau},\mathbf{I})\gamma(\mathbf{I},\tau^{-1}\sigma\tau)(\tilde{\pi}(e_{i_{\tau^{-1}\sigma\tau(1)}i_{1}}) \cdots \tilde{\pi}(e_{i_{\tau^{-1}\sigma\tau(k)}i_{k}}))\\
=&\frac{1}{k!}\mathop{\sum}\limits_{\tau\in \mathfrak{S}_k} \mathfrak{z}_{\tau^{-1}\sigma^{-1}\tau},
\end{align*}
where the validity of the second equality is determined by  \eqref{eq::gammap}, the third holds by \eqref{gamma} along with  \eqref{psigma}, and the fourth is established by replacing $\mathbf{I}_{\tau}$ by $\mathbf{I}$. 
\end{proof}

\begin{proposition}\label{sigmacycle}
Suppose that $\sigma\in \mathfrak{S}_k$ is a permutation, and $\sigma=\sigma_1\cdots\sigma_s$ is a product of disjoint cycles. Then $\eta\circ \pi(\theta_{\sigma})=\eta\circ \pi(\theta_{\sigma_1})\eta\circ \pi(\theta_{\sigma_2})\cdots \eta\circ \pi(\theta_{\sigma_s})$.
\end{proposition}
\begin{proof}
Let $\tau=\sigma^{-1}$ and $\tau_i=\sigma_i^{-1}$ for $i=1,\cdots,s$, then $\tau=\tau_1\cdots \tau_s$ is a product of disjoint cycles. By \eqref{sign}, $$\eta\circ \pi(\theta_{\tau^{-1}})=\mathop{\sum}\limits_{\mathbf{I}}p(\mathbf{I}+\mathbf{I}_{\tau},\mathbf{I})\gamma(\mathbf{I},\tau) \tilde{\pi}(e_{i_{\tau(1)},i_1})\tilde{\pi}(e_{i_{\tau(2)},i_2})\cdots \tilde{\pi}(e_{i_{\tau(k)},i_k})
\in S(\mathfrak{g}).$$
Suppose that $\tau_r=(j_{r,1}j_{r,2}\cdots j_{r,t_r})$ and let $\mathbf{I}_r=(|i_{j_{r_1}}|,|i_{j_{r,2}}|,\cdots,|i_{j_{r,t_r}}|)$. Then $\tau_r$ can be regarded  as a permutation in $\{j_{r,1},\cdots,j_{r,t_r}\}$ and $\eta\circ \pi(\theta_{\tau_r^{-1}})$ is equal to 
$$\mathop{\sum}\limits_{\mathbf{I}_{r}}p(\mathbf{I}_r+\tau_r^{-1}\cdot \mathbf{I}_r,\mathbf{I}_r)\gamma(\mathbf{I}_r,\tau_r) \tilde{\pi}(e_{i_{\tau_r(j_{r,1})},i_{j_{r,1}}})\tilde{\pi}(e_{i_{\tau_r(j_{r,2})},i_{j_{r,2}}})\cdots \tilde{\pi}(e_{i_{\tau_r(j_{r,t_r})},i_{j_{r,t_r}}}). $$

Let $A$ be a free commutative superalgebra generated by $a_1,\cdots,a_k$ and $b_1,\cdots,b_k$ with $|a_p|=|b_p|=|e_{i_p}|$  for all $p=1,\cdots,k$. Then, 
\begin{align*}
a_1b_1\cdots a_kb_k=&p(\mathbf{I}+\mathbf{I}_{\tau},\mathbf{I})\gamma(\mathbf{I},\tau)a_{\tau(1)}b_1\cdots a_{\tau(k)}b_k,\\
a_{j_{r,1}}b_{j_{r,1}}\cdots a_{j_{r,t_r}}b_{j_{r,t_r}}
=&p(\mathbf{I}_r+\tau_r^{-1}\cdot \mathbf{I}_r,\mathbf{I}_r)\gamma(\mathbf{I}_r,\tau_r) a_{\tau_r(j_{r,1})}b_{j_{r,1}}\cdots a_{\tau(j_{r,t_r})}b_{j_{r,t_r}},
\end{align*}
for all $r=1,\cdots, s$.
Note that $a_1b_1\cdots a_kb_k=\mathop{\prod}\limits_{r=1}^s a_{j_{r,1}}b_{j_{r,1}}\cdots a_{j_{r,t_r}}b_{j_{r,t_r}}$ and $S(\mathfrak{g})$ can be viewed as a subalgebra of $A$ by the inclusion $\mathfrak{j}\colon \tilde{\pi}(e_{ij})\mapsto a_ib_j$. Therefore, we obtain
\begin{align*}
\mathfrak{j}(\eta\circ\pi (\theta_{\sigma}))&=\mathop{\sum}\limits_{\mathbf{I}}p(\mathbf{I}+\mathbf{I}_{\tau},\mathbf{I})\gamma(\mathbf{I},\tau)a_{\tau(1)}b_1\cdots a_{\tau(k)}b_k\\
&=\mathop{\sum}\limits_{\mathbf{I}}a_1b_1\cdots a_kb_k=\mathop{\prod}\limits_{r=1}^s\mathop{\sum}\limits_{\mathbf{I}_r} a_{j_{r,1}}b_{j_{r,1}}\cdots a_{j_{r,t_r}}b_{j_{r,t_r}}\\
&=\mathop{\prod}\limits_{r=1}^s \mathop{\sum}\limits_{\mathbf{I}_r}p(\mathbf{I}_r+\tau_r^{-1}\cdot \mathbf{I}_r,\mathbf{I}_r)\gamma(\mathbf{I}_r,\tau_r) a_{\tau_r(j_{r,1})}b_{j_{r,1}}\cdots a_{\tau(j_{r,t_r})}b_{j_{r,t_r}}\\
&=\mathop{\prod}\limits_{r=1}^s \mathfrak{j}\eta\circ\pi(\theta_{\tau_r^{-1}}))=\mathfrak{j}\left(\eta\circ\pi(\theta_{\sigma_1})\eta\circ\pi(\theta_{\sigma_2})\cdots \eta\circ\pi(\theta_{\sigma_s})\right).
\end{align*}
This completes the proof since $\mathfrak{j}$ is an inclusion.
\end{proof}
\begin{corollary}
The elements $\eta\circ\pi(\theta_{(12\cdots k)}) $ with all $ k\geqslant 1 $ generate the algebra $S(\mathfrak{gl}_{m|n})^{\mathfrak{gl}_{m|n}}$.
\end{corollary}
\subsection{Explicit invariants for $\mathfrak{gl}_{m|n}$}\label{subsectiongl}

By \eqref{split} and \eqref{psi}, we know that $\eta\colon T(\mathfrak{gl}_{m|n})\to S(\mathfrak{gl}_{m|n})$  is a split $\mathfrak{gl}_{m|n}$-module homomorphism, and $\psi \colon S(\mathfrak{gl}_{m|n})\to \mathrm{U}(\mathfrak{gl}_{m|n})$ is the $\mathfrak{gl}_{m|n}$-module isomorphism. Hence, the set $\big\{\psi\circ\eta \circ\pi(\theta_{\sigma^{-1}})\big|\sigma\in \mathfrak{S}_k,k\in \mathbb{Z}_+\big\}$ spans $\mathcal{Z}(\mathfrak{gl}_{m|n})$. Consequently, $\mathcal{Z}(\mathfrak{gl}_{m|n})$ can be spanned by  $\big\{\mathfrak{z}_{\sigma}\big|\sigma\in \mathfrak{S}_k,k\in \mathbb{Z}_+\big\}$ according to Theorem \ref{etaeta'}. 

Molev systematically investigated the Gelfand invariants for classical Lie algebras $\mathfrak{gl}_{n}, \mathfrak{o}_{n}$ and $\mathfrak{sp}_{2n}$ and obtained the generators for the center of the universal enveloping algebras, as documented in \cite[Chapter 4, 5]{M18}. The related techniques can also be extended to classical Lie superalgebras $\mathfrak{gl}_{m|n}$ and $\mathfrak{osp}_{m|2n}$. More specifically, 
the basis element $E_{ij}$ of $\mathfrak{gl}_{m|n}$ can be regarded as generators of the
universal enveloping algebra $\mathrm{U}(\mathfrak{gl}_{m|n})$ with the following relation:
\begin{equation}\label{glrelation}
E_{ij}E_{kl}-(-1)^{(|i|+|j|)(|k|+|l|)}E_{kl}E_{ij}=\delta_{jk}E_{il}-(-1)^{(|i|+|j|)(|k|+|l|)}\delta_{li}E_{kj}.
\end{equation}

Let $\widehat{E}=\left[(-1)^{|i|}E_{ij}\right]= \mathop{\sum}\limits_{i,j=1}^{m+n} (-1)^{|i||j|+|i|+|j|}e_{ij}\otimes E_{ij}\in \mathrm{End}(V)\otimes \mathrm{U}(\mathfrak{gl}_{m|n})$,where $e_{ij}\in \mathrm{End}(V)$ denotes the standard matrix unit. The supertranspose of $\mathrm{End}(V)$, defined by $e_{ij}^{\mathrm{st}}=(-1)^{(|i|+|j|)|i|}e_{ji}$, is an anti-automorphism, and it can be extend an anti-automorphism of $\mathrm{End}(V)^{\otimes k}$, which we denote by $S^{\mathrm{st}}$, where $S\in \mathrm{End}(V)^{\otimes k}$. We denote $E^{\mathrm{st}}$ as the supertranspose of $E$, that is,  
\begin{align*}
\widehat{E}^{\mathrm{st}}=\mathop{\sum}\limits_{i,j=1}^{m+n} (-1)^{|i||j|+|i|+|j|}e_{ij}^{\mathrm{st}}\otimes E_{ij}=\mathop{\sum}\limits_{i,j=1}^{m+n} (-1)^{|j|}e_{ji}\otimes E_{ij}.
\end{align*}

The relation \eqref{glrelation} can be written in the form
\begin{align}\label{eq:glpresent}
\widehat{E}_1\widehat{E}_2-\widehat{E}_2\widehat{E}_1=P\widehat{E}_2-\widehat{E}_2P, 
\end{align}
where $P=\mathop{\sum}\limits_{i,j=1}^{m+n}(-1)^{|j|}e_{ij}\otimes e_{ji}\in \mathrm{End}(V)^{\otimes 2}$ is known as the super transposition. The classical Gelfand invariants of the general linear Lie superalgebra are $\mathrm{Str}\widehat{E}^{k}$ for $k\geqslant 1$, as mentioned in \cite[Chapter 4]{M18} for the case of a simple Lie algebra.
Suppose that
\begin{align*}
S=\mathop{\sum}\limits_{\mathbf{I},\mathbf{J}}a_{j_1,\cdots,j_k}^{i_1,\cdots,i_k}e_{i_1,j_1}\otimes \cdots\otimes e_{i_k,j_k}\in[\mathrm{End}(V)^{\otimes k}]^{\mathfrak{gl}_{m|n}},
\end{align*}
then 
\begin{align*}
\eta'\circ \pi(S)=\mathop{\sum}\limits_{\mathbf{I},\mathbf{J}}a_{j_1,\cdots,j_k}^{i_1,\cdots,i_k}E_{i_1,j_1} \cdots E_{i_k,j_k}=\mathrm{Str}_{1,\cdots,k}(\widehat{E}_1^{\mathrm{st}}\cdots \widehat{E}_k^{\mathrm{st}}S)\in \mathcal{Z}(\mathfrak{gl}_{m|n}).
\end{align*}
Thus, 
\begin{align}\label{strgl}
\eta'\circ \pi(S^{\mathrm{st}})=&\mathop{\sum}\limits_{\mathbf{I},\mathbf{J}}a_{j_1,\cdots,j_k}^{i_1,\cdots,i_k}(-1)^{(|i_1|+|j_1|)|i_1|+\cdots+(|i_k|+|j_k|)|i_k| } E_{j_1,i_1} \cdots E_{j_k,i_k}\nonumber\\ 
=&\mathrm{Str}_{1,\cdots,k}(\widehat{E}_1^{\mathrm{st}}\cdots \widehat{E}_k^{\mathrm{st}}S^{\mathrm{st}})=\mathrm{Str}_{1,\cdots,k}(\widehat{E}_1\cdots \widehat{E}_kS) \in\mathcal{Z}(\mathfrak{gl}_{m|n}).
\end{align}

By a similar proof of \cite[Theorem 4.5.1]{M18}, we obtain the following theorem.
\begin{theorem}\label{Gelfandgl}
For any $\mathfrak{gl}_{m|n}$-invariant $S\in \mathrm{End}(V)^{\otimes k}$ and $u_1,\cdots,u_k\in \mathbb{C}$, then the element
\begin{align*}
\mathrm{Str}_{1,\cdots,k}(u_1+\widehat{E}_1)\cdots (u_k+\widehat{E}_k)S
\end{align*}
belongs to $\mathcal{Z}(\mathfrak{gl}_{m|n})$.
\end{theorem}

The following remark states that the classical Gelfand invariants constructed from the matrix presentation coincide with the invariants from the Schur-Weyl duality if we set $\sigma^{-1}=(12\cdots k)\in\mathfrak{S}_k$.
\begin{remark}
The classical Gelfand invariants  $\mathrm{Str}\widehat{E}^{k} $ are equal to  $\eta'\circ\pi(\theta_{\sigma^{-1}})$, where $\sigma^{-1}=(12\cdots k)$. Indeed, let $P_{\sigma}=P_{k-1,k}\cdots P_{23}P_{12}$, then 
$$P_{\sigma}=\mathop{\sum}\limits_{\mathbf{I}}(-1)^{|i_k|+|i_1||i_k|+|i_2||i_1|+\cdots +|i_{k}||i_{k-1}|}e_{i_1,i_k}\otimes e_{i_2,i_1}\otimes\cdots\otimes e_{i_k,i_{k-1}}.$$
Then, we have
\begin{align*}
\mathfrak{z}_{\sigma^{-1}}=\mathop{\sum}\limits_{\mathbf{I}}(-1)^{|i_1|+|i_2|+\cdots+|i_{k-1}| }E_{i_ki_1}E_{i_1i_2}\cdots E_{i_{k-1}i_k}
=\mathrm{Str}_{1,\cdots,k}P_{\sigma}\widehat{E}_{1}\widehat{E}_2\cdots \widehat{E}_k=\mathrm{Str}\widehat{E}^{k}.
\end{align*}
\end{remark}

\begin{corollary}
The elements $\mathrm{Str}\widehat{E}^k $ with $ k\geqslant 1 $ generate the center $\mathcal{Z}(\mathfrak{gl}_{m|n})$ of the universal enveloping algebra $\mathrm{U}(\mathfrak{gl}_{m|n})$.
\end{corollary}
\begin{proof}
The top degree component of the polynomial $\gamma_{-\rho}\circ\zeta(\mathrm{Str}\widehat{E}^k)$ is equal to $h_1^k+\cdots +h_m^k+(-1)^{k-1}(h_{m+1}^k+\cdots+h_{m+n}^k)$. Therefore, the Gelfand invariants $\mathrm{Str}\widehat{E}^k$ with $k\geqslant 1$ generate the algebra $\mathcal{Z}(\mathfrak{gl}_{m|n})$.
\end{proof}
\subsection{Explicit invariants for $\mathfrak{q}_{n}$}\label{subsectionq}

Note that $\mathfrak{q}_n$ is also an associative superalgebra with multiplication
\begin{align*}
    \begin{pmatrix}
        A&B\\B&A
    \end{pmatrix}\cdot\begin{pmatrix}
        C&D\\D&C
    \end{pmatrix}=\begin{pmatrix}
        AC+BD&AD+BC\\BC+AD&BD+AC
    \end{pmatrix},
\end{align*}
for all $A,B,C,D$ are $n\times n$-matrices. Moreover, the map $\tilde{\pi}^{\otimes k}\circ\Omega_k$ is a $\mathrm{U}(\mathfrak{q}_{n})$-module superalgebra homomorphism from $\mathrm{End}(V^{\otimes k})$ to $\mathfrak{q}_{n}^{\otimes k}$, where the multiplication of $\mathfrak{q}_{n}^{\otimes k}$ is induced by the multiplication of $\mathfrak{q}_{n}$ as defined above. Composing the map with $\Psi_k$, we obtain a surjective algebra homomorphism from Hecke-Clifford algebra $\mathcal{H}_k$ to $\left(\mathfrak{q}_{n}^{\otimes k}\right)^{\mathfrak{q}_{n}}$. 

Since $\tilde{\pi}(\mathcal{P})=0$, we have 
$$\tilde{\pi}^{\otimes k}\circ\Omega_k\circ \Psi_k(c_r)=\mathop{\sum}\limits_{\mathbf{I}}  \tilde{\pi}(e_{i_1i_1})\otimes \cdots\otimes \tilde{\pi}( e_{i_{r-1}i_{r-1}})\otimes \tilde{\pi}(\mathcal{P})\otimes \tilde{\pi}(e_{i_ri_r})\otimes\cdots\otimes \tilde{\pi}(e_{i_ki_k})=0,$$
for $1\leqslant r\leqslant k$. Therefore,  we have the following proposition.
\begin{proposition}
For every $\sigma\in\mathfrak{S}_k$, $\mathbf{I}=(i_1,\ldots,i_k)$, then 
$$ \pi(\theta_{\sigma^{-1}})=\frac{1}{2^k}\mathop{\sum}\limits_{\mathbf{I}} p(\mathbf{I}+\mathbf{I}_{\sigma},\mathbf{I})\gamma(\mathbf{I},\sigma) H_{i_{\sigma(1)},i_1}\otimes \cdots\otimes H_{i_{\sigma(k)},i_k}$$ 
is a $\mathfrak{q}_n$-invariant in $\mathfrak{q}_n^{\otimes k}$ and these elements span $\left(\mathfrak{q}_n^{\otimes k}\right)^{\mathfrak{q}_n}$.
\end{proposition}

We conclude that the set $\left\{\psi\circ\eta\circ\pi(\theta_{\sigma})\big|\sigma\in \mathfrak{S}_k,k\in \mathbb{Z}_+\right\}$ spans the center $\mathcal{Z}(\mathfrak{q}_n)$. Here, $\eta\colon T(\mathfrak{q}_n)\longrightarrow  S(\mathfrak{q}_n)$  is the split $\mathfrak{q}_n$-module homomorphism and $\psi\colon S(\mathfrak{q}_n)\longrightarrow  \mathrm{U}(\mathfrak{q}_n)$ is the $\mathfrak{q}_n$-module isomorphism, as shown in \eqref{split} and \eqref{psi}. According to 
Theorem \ref{etaeta'}, the $\mathfrak{q}_n$-invariant in $\mathrm{U}(\mathfrak{q}_n)$ can be spanned by the set $\big\{\mathfrak{z}_{\sigma}\big|\sigma\in \mathfrak{S}_k,k\in \mathbb{Z}_+\big\}$. Therefore, the restriction of $\eta'$ on $T(\mathfrak{q}_n)^{\mathfrak{q}_n}$ is surjective.

\begin{remark}\label{gelfandq}
If $\sigma=(12\cdots k)$, then
$$\mathfrak{z}_{\sigma}=\frac{1}{2^k}\mathop{\sum}\limits_{\mathbf{I}}(-1)^{|i_1|+|i_2|+\cdots+|i_{k-1}| }H_{i_ki_1}H_{i_1i_2}\cdots H_{i_{k-1}i_k}$$
by Example \ref{examplegelfand}. Note that $H_{ij}=H_{-i,-j}$ for all $i,j$, hence
\begin{align*}
\mathfrak{z}_{\sigma}=&\frac{1}{2^k}\mathop{\sum}\limits_{\mathbf{I}}(-1)^{|-i_1|+|-i_2|+\cdots+|-i_{k-1}| }H_{-i_k,-i_1}H_{-i_1,-i_2}\cdots H_{-i_{k-1},-i_k}\\
=&\frac{(-1)^{k-1}}{2^k}\mathop{\sum}\limits_{\mathbf{I}}(-1)^{|i_1|+|i_2|+\cdots+|i_{k-1}| }H_{i_k,i_1}H_{i_1,i_2}\cdots H_{i_{k-1},i_k}\\
=&(-1)^{k-1}\mathfrak{z}_{\sigma}.
\end{align*}
This means $\mathfrak{z}_{\sigma}=0$ if $k$ is even.
\end{remark}
In \cite{Se83}, the author asserted that the generators of $\mathcal{Z}(\mathfrak{q}_n)$,  which has been demonstrated by Sergeev and Nazarov in \cite{NS06}. Subsequently, we will proceed to compare his results with ours. Let 
$$e_{ij}=e_{ij}^{(1)}=E_{ij}+E_{-i,-j},\quad f_{ij}=f_{ij}^{(1)}=E_{i,-j}+E_{-i,j}$$ 
as the standard bases of $\mathfrak{q}_n$. The elements $e_{ij}^{(m)}$ and $f_{ij}^{(m)}$ that belong to $\mathrm{U}(\mathfrak{q}_n)$ are defined by induction as follows:
\begin{align*}
e_{ij}^{(m)}&=\mathop{\sum}\limits_{i=1}^n e_{ik}e_{kj}^{(m-1)}+(-1)^{m-1}\mathop{\sum}\limits_{i=1}^n f_{ik}f_{kj}^{(m-1)},\\
f_{ij}^{(m)}&=\mathop{\sum}\limits_{i=1}^n e_{ik}f_{kj}^{(m-1)}+(-1)^{m-1}\mathop{\sum}\limits_{i=1}^n f_{ik}e_{kj}^{(m-1)}.
\end{align*}
Then $\mathcal{Z}_{k}=\mathop{\sum}\limits_{i=1}^ne_{ii}^{(k)}$ with odd $k$ belong to $\mathcal{Z}(\mathfrak{g})$ and these elements generate the algebra $\mathcal{Z}(\mathfrak{g})$.

\begin{proposition}
Let $k$ be an odd integer and $\sigma=(12\cdots k)$, then  $\mathcal{Z}_k=2^{k-1}\mathfrak{z}_{\sigma}$.
\end{proposition}
\begin{proof}
Denote by \begin{align*}
E_{i_0i_m}^{(m)}&=\mathop{\sum}\limits_{i_1,\cdots,i_{m-1}}(-1)^{|i_1|+|i_2|+\cdots+|i_{m-1}|} H_{i_0i_1}H_{i_1i_2}\cdots H_{i_{m-2}i_{m-1}}H_{i_{m-1}i_m},\\
F_{i_0i_m}^{(m)}&=\mathop{\sum}\limits_{i_1,\cdots,i_{m-1}}(-1)^{|i_1|+|i_2|+\cdots+|i_{m-1}|} H_{i_0i_1}H_{i_1i_2}\cdots H_{i_{m-2}i_{m-1}}H_{i_{m-1},-i_m},
\end{align*}
for $1\leqslant i_0,i_m\leqslant n$. Then
\begin{align*}
E_{i_0i_{m+1}}^{(m+1)}=&\mathop{\sum}\limits_{i_1,\cdots,i_{m}}(-1)^{|i_1|+|i_2|+\cdots+|i_{m}|} H_{i_0i_1}H_{i_1i_2}\cdots H_{i_{m-1}i_{m}}H_{i_{m}i_{m+1}}\\
=&e_{i_0i_1}\mathop{\sum}\limits_{\substack{i_2,\cdots,i_{m}\\i_1>0}}(-1)^{|i_2|+\cdots+|i_{m}|} H_{i_1i_2}\cdots H_{i_{m-1}i_{m}}H_{i_{m}i_{m+1}}\\
&-f_{i_0i_1}\mathop{\sum}\limits_{\substack{i_2,\cdots,i_{m}\\i_1<0}}(-1)^{|i_2|+\cdots+|i_{m}|} H_{i_1i_2}\cdots H_{i_{m-1}i_{m}}H_{i_{m}i_{m+1}}\\
=&e_{i_0i_1}E_{i_1i_{m+1}}^{(m)}-f_{i_0i_1}\mathop{\sum}\limits_{\substack{i_2,\cdots,i_{m}\\i_1>0}}(-1)^{|-i_2|+\cdots+|-i_{m}|} H_{i_1i_2}\cdots H_{i_{m-1}i_{m}}H_{i_{m},-i_{m+1}}\\
=&e_{i_0i_1}E_{i_1i_{m+1}}^{(m)}+(-1)^{m}f_{i_0i_1}F_{i_1i_{m+1}}^{(m)}.
\end{align*}
Similarly,
\begin{align*}
F_{i_0i_{m+1}}^{(m+1)}=e_{i_0i_1}F_{i_1i_{m+1}}^{(m)}+(-1)^{m}f_{i_0i_1}E_{i_1i_{m+1}}^{(m)}.
\end{align*}
Therefore, the recursive relation of $E_{ij}^{(m)}$ and $F_{ij}^{(m)}$ is the same as $e_{ij}^{(m)}$ and $f_{ij}^{(m)}$, respectively. Thus, $E_{ij}^{(m)}=e_{ij}^{(m)}$ (resp. $F_{ij}^{(m)}=f_{ij}^{(m)}$) holds for all $1\leqslant i,j\leqslant n$ and $m\in\mathbb{Z}_{>0}$. This means
\begin{align*}
\mathcal{Z}_k=\mathop{\sum}\limits_{i=1}^ne_{ii}^{(k)}=\mathop{\sum}\limits_{i=1}^nE_{ii}^{(k)}=2^{k-1}\mathfrak{z}_{\sigma}.
\end{align*}
by Remark \ref{gelfandq}.
\end{proof}

If $k$ is odd, then the top degree component of the image of the Harish-Chandra homomorphism of $\mathcal{Z}_k$ is equal to $h_1^k+\cdots +h_n^k$. Therefore, the central elements $\mathcal{Z}_k$ with $k$ odd generate the center $\mathcal{Z}(\mathfrak{q}_n)$.

\subsection{Explicit invariants for $\mathfrak{osp}_{m|2n}$}\label{subsectionosp}
Since the vector representation of $\mathfrak{osp}_{m|2n}$ is self-dual, thus $V^{\otimes 2k}$ is isomorphic to $\mathrm{End}(V^{\otimes k})$ as $\mathfrak{osp}_{m|2n}$-modules. Therefore, there is a bijection between the $\mathfrak{g}$-invariants of $V^{\otimes 2k}$ and the $\mathfrak{g}$-invariants of $\mathrm{End}(V^{\otimes k})$. 

Note that $c=\mathop{\sum}\limits_{i}\varepsilon_{i'}e_i\otimes e_{i'}\in V^{\otimes 2}$ is an $\mathfrak{osp}_{m|2n}$-invariant, so does $$c^{\otimes k}=\mathop{\sum}\limits_{\mathbf{I}}\varepsilon_{i_1'}\cdots\varepsilon_{i_k'}e_{i_1}\otimes e_{i_1'}\otimes\cdots\otimes e_{i_k}\otimes e_{i_k'}.$$
By \cite[Section 3]{DLZ18}, the set $\{\sigma\cdot c^{\otimes k}\}_{\sigma\in\mathfrak{S}_{2k}}$ spans $\big(V^{\otimes 2k}\big)^{\mathrm{OSP}(V)}$, where $\mathrm{OSP}(V)$ is the ortho-symplectic supergroup. Denote $\overline{g}\in \mathfrak{S}_{2k}$ by
\begin{align}\label{tilde{g}}
    \overline{g}(2s-1)=2g(s)-1 \quad \text{and}\quad \overline{g}(2s)=2g(s),
\end{align}
for all  $g\in\mathfrak{S}_k$ and $s=1,\cdots,k$.  

Let $K$ be the subgroup of $\mathfrak{S}_{2k}$ generated by the swaps $(12), (34),\cdots, (2k-1,2k)$ and $H$ be the subgroup of $\mathfrak{S}_{2k}$ generated by elements in $K$ and $\overline{g}$ for all $g\in\mathfrak{S}_k$.

Obviously, $H$ acts trivially on $c^{\otimes 2k}$. Let $\mathrm{B}_k$ be a set of representatives of the left coset $\mathfrak{S}_{2k}/H$. Consequently, the set $\{\sigma\cdot c^{\otimes k}\}_{\mathrm{B}_k}$ spans $\big(V^{\otimes 2k}\big)^{\mathrm{OSP}(V)}$. The cardinality of $\mathrm{B}_k$ is $(2k-1)!!$, which is the same as the dimension of the Brauer algebra $\mathcal{B}_k(m-2n)$. 

Let 
$$W=\left\{(j_1,\cdots,j_{2k})|~~ j_1=j_2',j_3=j_4',\cdots,j_{2k-1}=j_{2k}' \right\}.$$
Now, we define $\mathbf{J}^{o}=(j_1,j_3,\cdots,j_{2k-1})$ and $\mathbf{J}^{e}=(j_2,j_4,\cdots,j_{2k})$, and denote $\varepsilon(\mathbf{J}^{e})$ by $\varepsilon_{j_{2}}\varepsilon_{j_{4}}\cdots\varepsilon_{j_{2k}}$, then $c^{\otimes k}=\mathop{\sum}\limits_{J\in W}\varepsilon(\mathbf{J}^{e})e_{j_1}\otimes e_{j_2}\otimes \cdots\otimes e_{j_{2k}}$, and  
\begin{align}\label{sigmack}
\sigma^{-1}\cdot c^{\otimes k}=\mathop{\sum}\limits_{\mathbf{J}\in W}\varepsilon(\mathbf{J}^{e}) \gamma(\mathbf{J},\sigma)e_{j_{\sigma(1)}}\otimes \cdots\otimes e_{j_{\sigma(2k)}}.
\end{align}

Recall that $V$ is isomorphic to $V^*$ through the map $\Theta\colon e_i\mapsto \varepsilon_ie_{i'}^*$. Now, we present a procedure for converting the element $\sigma^{-1}\cdot c^{\otimes k}$ into $[\mathrm{End}(V)^{\otimes k}]^{\mathfrak{g}}$.
\begin{itemize}
\item Step 1: Apply the isomorphism $\Theta$ to any $k$ terms of \eqref{sigmack}.
\item Step 2: Rearrange the order such that the odd terms belong to $V$ and the even terms belong to $V^*$.
\item Step 3: Calculate the image through the morphisms $$(V\otimes V^*)^{\otimes k}\longrightarrow \mathrm{End}(V)^{\otimes k}.$$
\end{itemize}
Here, we choose a specific procedure. First, apply the isomorphism $\Theta$ to the even terms of \eqref{sigmack}, and calculate the image through the morphism $$(V\otimes V^*)^{\otimes k}\longrightarrow \mathrm{End}(V)^{\otimes k},$$ which is denoted by $\theta_{\sigma^{-1}}$.
Then 
\begin{align*}
\theta_{\sigma^{-1}}=\mathop{\sum}\limits_{\mathbf{J}\in W}\varepsilon(\mathbf{J}^{e})\varepsilon(\mathbf{J}_{\sigma}^e)\gamma(\mathbf{J},\sigma)e_{j_{\sigma(1)}j_{\sigma(2)}'}\otimes \cdots \otimes e_{j_{\sigma(2k-1)}j_{\sigma(2k)}'}.
\end{align*}
\begin{remark}\label{sigmaBrauer}
\begin{itemize}
    \item[(1).] The element $c^{\otimes k}$ can be interpreted as an element of $\mathrm{Hom}(\mathbb{C},V^{\otimes 2k})$. It can be visually represented by a $(0,2k)$-Brauer diagram, which consists of a horizontal line with $2k$ dots and $k$ arcs:
$$
\begin{tikzpicture}[scale=1,thick,>=angle 90]
\begin{scope}[xshift=4cm]
\draw (2.8,1) to  [out=-90,in=-90] +(0.6,0);
\draw (4,1) to [out=-90, in=-90] +(0.6,0);
\node at (5.5,0.9) {$\cdots$};
\draw (6.4,1) to [out=-90, in=-90] +(0.6,0);
\end{scope}.
\end{tikzpicture}
$$ 
The special procedure can be interpreted as shifting all even points downwards and moving one step to the left. For example, the above Brauer diagram turns into the following $(k,k)$-Brauer diagram through the special procedure:
$$ 
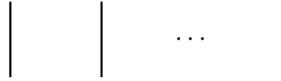
\begin{tikzpicture}[scale=1,thick,>=angle 90]
\begin{scope}[xshift=4cm]
\draw (2.8,1) to  [out=-90,in=90] +(0,-1);
\draw (4,1) to [out=-90, in=90] +(0,-1);
\node at (5.2,0.5) {$\cdots$};
\draw (6.4,1) to [out=-90, in=90] +(0,-1);
\end{scope}
\end{tikzpicture}
$$ 
\item[(2).] The element $\sigma\cdot c^{\otimes k}$ can be represented by a $(0,2k)$-Brauer diagram, which consists of $2k$ dots in a horizontal line and $k$ arcs, and the $s$-th arc connects $\sigma(2s-1)$-dot and $\sigma(2s)$-th dot for $s=1,2,\cdots,k$.
Through the special procedure, the  $(0,2k)$-Brauer diagram turn into the a $(k,k)$-Brauer diagram, denoted by $d_{\sigma}$. For further information regarding the Brauer diagram, we refer the reader to \cite{LZ12, LZ15, LZ24}.

For example, if $\sigma=(234)(56)\in\mathfrak{S}_6$, then $\sigma\cdot c^{\otimes k}$ can be represented by $(0,2k)$-Brauer diagram
$$
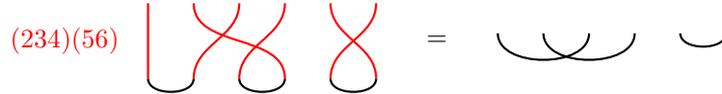
\begin{tikzpicture}[scale=1,thick,>=angle 90]
\begin{scope}[xshift=4cm]
\node at (-3.5,0.5)[color=red] {$(234)(56)$};
\draw (-2.4,0) to  [out=-90,in=-90] +(0.6,0);
\draw (-1.2,0) to [out=-90, in=-90] +(0.6,0);
\draw (0,0) to [out=-90, in=-90] +(0.6,0);
\draw (-2.4,0)[color=red] to [out=90,in=-90] +(0,1);
\draw (0.6,0)[color=red] to [out=90,in=-90] +(-0.6,1);
\draw (-1.8,0)[color=red] to [out=90,in=-90] +(0.6,1);
\draw (-1.2,0)[color=red] to [out=90,in=-90] +(0.6,1);
\draw (-0.6,0)[color=red] to [out=90,in=-90] +(-1.2,1);
\draw (0,0)[color=red] to [out=90,in=-90] +(0.6,1);
\node at (1.4,0.5) {$=$};
\draw (2.2,0.6) to  [out=-90,in=-90] +(1.2,0);
\draw (2.8,0.6) to [out=-90, in=-90] +(1.2,0);
\draw (4.6,0.6) to [out=-90, in=-90] +(+0.6,0);
\end{scope}
\end{tikzpicture}
$$
and it turn into the $(k,k)$-Brauer diagram
$$
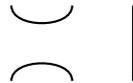
\begin{tikzpicture}[scale=1,thick,>=angle 90]
\begin{scope}[xshift=4cm]
\draw (2.2,1) to  [out=-90,in=-90] +(0.8,0);
\draw (2.2,0) to [out=90, in=90] +(0.8,0);
\draw (3.8,1) to [out=-90, in=90] +(0,-1);
\end{scope}
\end{tikzpicture}
$$
\end{itemize}
\end{remark}

\begin{lemma}\label{tildetau}
For every $g\in\mathfrak{S}_k$, we have $\gamma(\mathbf{J}^{o}+\mathbf{J}^{e},g)=\gamma(\mathbf{J},\overline{g})$.
\end{lemma}
\begin{proof}
Suppose $A$ is a free commutative superalgebra generated by $a_1,\cdots,a_{2k}$ with $|a_i|=|j_i|$ for all $1\leqslant i\leqslant 2k$. Then
$$a_1a_2\cdots a_{2k}=\gamma(\mathbf{J},\sigma)a_{\sigma(1)}a_{\sigma(2)}\cdots a_{\sigma(2k)},$$
for all $\sigma\in \mathfrak{S}_{2k}$. Let $b_i=a_{2i-1}a_{2i}$ for all $1\leqslant i\leqslant k$ and $\mathbf{J}'=(j_1',\cdots,j_{k}')$ with $|j_i'|=|b_i|$. Then $|b_i|=|a_{2i-1}|+|a_{2i}|$ and hence $|\mathbf{J}'|=|\mathbf{J}^{o}|+|\mathbf{J}^e|$. Therefore,
\begin{align*}
&a_1a_2\cdots a_{2k}=b_1b_2\cdots b_k =\gamma(\mathbf{J}',g)b_{g(1)}b_{g(2)}\cdots b_{g(k)}\\
=&\gamma(\mathbf{J}^o+\mathbf{J}^e,g)a_{2g(1)-1}a_{2g(1)} a_{2g(2)-1}a_{2g(2)}\cdots a_{2g(k)-1}a_{2g(k)}\\
=&\gamma(\mathbf{J}^o+\mathbf{J}^e,g)a_{\overline{g}(1)}a_{\overline{g}(2)}\cdots a_{\overline{g}(2k)},
\end{align*}
for all $g\in \mathfrak{S}_k$.
This means $\gamma(\mathbf{J}^o+\mathbf{J}^e,g)=\gamma(\mathbf{J},\overline{g})$ for all $g\in \mathfrak{S}_k$.
\end{proof}
Recall that $\mathfrak{z}_{\sigma}=\eta'\circ\pi(\theta_{\sigma})$ for all $\sigma\in \mathfrak{S}_{2k}$.
\begin{proposition}\label{osprelation}
The equation $\psi\circ\eta\circ\pi (\theta_{\sigma})=\frac{1}{k!}\mathop{\sum}\limits_{\tau\in\mathfrak{S}_k}\mathfrak{z}_{\overline{\tau}\sigma}$ holds in $\mathcal{Z}(\mathfrak{osp}_{m|2n})$ for all $\sigma\in \mathfrak{S}_{2k}$.
\end{proposition}
\begin{proof}
By direct calculation, we have 
\begin{align*}
&\psi\circ\eta\circ\pi(\theta_{\sigma^{-1}})\\
=&\mathop{\sum}\limits_{\mathbf{J}\in W}\varepsilon(\mathbf{J}^e)\varepsilon(\mathbf{J}_{\sigma}^e)\gamma(\mathbf{J},\sigma)\frac{1}{k!}\mathop{\sum}\limits_{\tau\in\mathfrak{S}_k}\gamma(\mathbf{J}_{\sigma}^o+\mathbf{J}_{\sigma}^e,\tau)
F_{j_{\sigma(2\tau(1)-1)}j_{\sigma(2\tau(1))}'}\otimes \cdots \otimes F_{j_{\sigma(2\tau(k)-1)}j_{\sigma(2\tau(k))}'}\\
=&\frac{1}{k!}\mathop{\sum}\limits_{\tau\in\mathfrak{S}_k}\mathop{\sum}\limits_{\mathbf{J}\in W}\varepsilon(\mathbf{J}^e)\varepsilon( \mathbf{J}_{\sigma\overline{\tau}}^e)\gamma(\mathbf{J},\sigma\overline{\tau}) F_{j_{\sigma\overline{\tau}(1)}j_{\sigma\overline{\tau}(2)}'}\otimes \cdots \otimes F_{j_{\sigma\overline{\tau}(2k-1)}j_{\sigma\overline{\tau}(2k)}'}\\
=&\frac{1}{k!}\mathop{\sum}\limits_{\tau\in\mathfrak{S}_k}\mathfrak{z}_{(\sigma\overline{\tau})^{-1}}.
\end{align*}
where the second equality holds by \eqref{gamma} and Lemma \ref{tildetau}.
\end{proof}
We conclude that the set $\big\{\psi\circ\eta\circ\pi(\theta_{\sigma})\big|\sigma\in \mathfrak{S}_k,k\in \mathbb{Z}_+\big\}$ spans the center $\mathcal{Z}(\mathfrak{osp}_{2m+1|2n})$. Here, $\eta\colon T(\mathfrak{osp}_{2m+1|2n})\longrightarrow S(\mathfrak{osp}_{2m+1|2n})$  is the split $\mathfrak{osp}_{2m+1|2n}$-module homomorphism and $\psi\colon S(\mathfrak{osp}_{2m+1|2n})\longrightarrow  \mathrm{U}(\mathfrak{osp}_{2m+1|2n})$ is the $\mathfrak{osp}_{2m+1|2n}$-module isomorphism, as shown in \eqref{split} and \eqref{psi}. According to 
Proposition \ref{osprelation},  the $\mathfrak{osp}_{2m+1|2n}$-invariant in $\mathrm{U}(\mathfrak{osp}_{2m+1|2n})$ can be spanned by the set $\big\{\mathfrak{z}_{\sigma}\big|\sigma\in \mathfrak{S}_k,k\in \mathbb{Z}_+\big\}$. Therefore, the restriction of $\eta'$ on $T(\mathfrak{osp}_{2m+1|2n})^{\mathfrak{osp}_{2m+1|2n}}$ is surjective. 

The rest of this subsection is devoting to establish the Gelfand invariants for ortho-symplectic Lie superalgebra $\mathfrak{osp}_{m|2n}$. Obviously, 
 $$F_{ij}=-(-1)^{|j|(|i|+|j|)}\varepsilon_i\varepsilon_jF_{j'i'},$$ 
 and the elements $F_{ij}$ satisfy the supercommutation relations
\begin{align*}
\left[F_{ij}, F_{kl}\right]
=\delta_{jk}F_{il}-(-1)^{(|i|+|j|)(|k|+|l|)}\delta_{il}F_{kj}-(-1)^{|j|(|i|+|j|)}\varepsilon_i\varepsilon_j\delta_{ki'}F_{j'l}+(-1)^{(|i|+|j|)|k|}\varepsilon_i\varepsilon_j\delta_{lj'}F_{ki'},
\end{align*}
which is equivalent to 
\begin{align}\label{eq:osppresent}
\widehat{F}_1\widehat{F}_2-\widehat{F}_2\widehat{F}_1=(P-Q)\widehat{F}_2-\widehat{F}_2(P-Q),
\end{align}
where $\widehat{F}=[(-1)^{|i|}F_{ij}]=\mathop{\sum}\limits_{i,j}(-1)^{|i||j|+|i|+|j|}F_{ij}$ and $Q=\mathop{\sum}\limits_{i,j}(-1)^{|i||j|+|i|+|j|}\varepsilon_i\varepsilon_j e_{ij}\otimes e_{i'j'}\in \mathrm{End}(V)^{\otimes 2}$. By a similar proof of \cite[Theorem 5.3.1]{M18}, we obtain the following theorem.
\begin{theorem}\label{Gelfandosp}
For any $\mathfrak{osp}_{m|2n}$-invariant $S\in \mathrm{End}(V)^{\otimes k}$ and $u_1,\cdots,u_k\in \mathbb{C}$. Then the element
\begin{align*}
\mathrm{Str}_{1,\cdots,k}(u_1+\widehat{F}_1)\cdots (u_k+\widehat{F}_k)S
\end{align*}
belongs to the center $\mathcal{Z}(\mathfrak{osp}_{m|2n})$.
\end{theorem}
If we let $u_1=\cdots=u_k=0$, then the element constructed in the above theorem coincides with $\eta'\circ \pi^{\otimes k}(S^{\mathrm{st}})$  up to a scalar by a similar computation with \eqref{strgl}.

The top degree component of the polynomial $\gamma_{-\rho}\circ\zeta(\mathrm{Str}\widehat{F}^{2k})$ coincides with $2(h_1^{2k}+\cdots +h_m^{2k})+(-1)^{k-1}2(h_{m+1}^{2k}+\cdots+h_{m+n}^{2k})$. Therefore, the Gelfand invariants $\mathrm{Str}\widehat{F}^{2k}$ with $k\geqslant 1$ generate the algebra $\mathrm{U}(\mathfrak{osp}_{m|2n})^{\mathrm{OSP}(V)}$.

\begin{corollary}
The elements $\mathrm{Str}\widehat{F}^{2k}$ with $k\geqslant 1$ generate the center of $\mathcal{Z}(\mathfrak{osp}_{2m+1|2n})$. 
\end{corollary}

\subsection{An alternative proof of Scheunert's result on invariants for $\mathfrak{p}_{n}$}\label{subsectionp}

Scheunert conducted a systematic investigation on the Casimir elements of classical Lie superalgebras. Specifically, he demonstrated that the center of the universal enveloping algebra $\mathrm{U}(\mathfrak{p}_n)$ is trivial by analyzing the invariant supersymmetric multilinear form and utilizing algebraic geometry techniques. Further details can be found in \cite[Proposition 3]{Sc87}. In this subsection, we can simplify the issue of  the triviality of the center $\mathcal{Z}(\mathfrak{p}_n)$ to a basic problem involving the symmetric group. What is even more intriguing is that we can resolve this problem by investigating the properties of the Brauer algebra. Thus, we provide an alternative elementary proof regarding the triviality of the center $\mathcal{Z}(\mathfrak{p}_n)$.

The quadratic element $c=\mathop{\sum}\limits_{i}(-1)^{|i|}e_i\otimes e_{i'}\in V^{\otimes 2}$ is $\mathfrak{p}_{n}$-invariant. So does 
\begin{align}\label{standardinvariant}
c^{\otimes k}=\mathop{\sum}\limits_{I}(-1)^{|i_1|+|i_2|+\cdots+|i_{k}|}e_{i_1}\otimes e_{i_1'}\otimes\cdots\otimes e_{i_k}\otimes e_{i_k'}.
\end{align}

The set $\{\sigma\cdot c^{\otimes k}\}_{\sigma\in\mathfrak{S}_{2k}}$ spans $\big(V^{\otimes 2k}\big)^{\mathfrak{p}_n}$, this follows from work of Deligne–Lehrer–Zhang \cite[Section 5.3]{DLZ18}, also see \cite[Lemma 8.1.4]{Co181}. 

Since $(12)\cdot c=-c$ and the action $\overline{\tau}$ on $c^{\otimes k}$ is invariant  for every $\tau\in \mathfrak{S}_k$, the set $\{\sigma \cdot c^{\otimes k}\}_{c\in \mathrm{B}_k}$ spans $\left(V^{\otimes 2k}\right)^{\mathfrak{p}_{n}}$ where $\mathrm{B}_k$ is a set of representatives of the left coset $\mathfrak{S}_{2k}/H$. Here, $H$ is subgroup of $\mathfrak{S}_{2k}$ defined in subsection \ref{subsectionosp}.

  Let 
$$W=\left\{(j_1,\cdots,j_n)| j_1=j_2',j_3=j_4',\cdots,j_{2k-1}=j_{2k}' \right\}.$$

Denote by $|\mathbf{J}^{o}|=|j_1|+|j_3|+\cdots+|j_{2k-1}|$, then $$c^{\otimes k}=\mathop{\sum}\limits_{\mathbf{J}\in W}(-1)^{|\mathbf{J}^{o}|}e_{j_1}\otimes e_{j_2}\otimes \cdots\otimes e_{j_{2k}}$$ and  
\begin{align}\label{sigmackp}
\sigma^{-1}\cdot c^{\otimes k}=\mathop{\sum}\limits_{\mathbf{J}\in W}(-1)^{|\mathbf{J}^{o}|} \gamma(\mathbf{J},\sigma)e_{j_{\sigma(1)}}\otimes \cdots\otimes e_{j_{\sigma(2k)}}.
\end{align}

Note that $V$ is isomorphic to $V^*$ through the map $\Theta\colon e_i\mapsto e_{i'}^*$. The procedure of transforming \eqref{sigmackp} into $\left[\mathrm{End}(V)^{\otimes k}\right]^{\mathfrak{g}}$ is analogous to the $\mathfrak{osp}_{m|2n}$ case. The $\mathfrak{g}$-invariants in $\mathrm{End}(V)^{\otimes k}$ by the special procedure are 
\begin{align}\label{pcenter}
\theta_{\sigma^{-1}}=\mathop{\sum}\limits_{\mathbf{J}\in W}
(-1)^{|\mathbf{J}^{o}|+|\mathbf{J}_{\sigma}^{o}|} p(\mathbf{1},\mathbf{J}_{\sigma}^o+\mathbf{J}_{\sigma}^e)\gamma(\mathbf{J},\sigma)e_{j_{\sigma(1)}j_{\sigma(2)}'}\otimes \cdots \otimes e_{j_{\sigma(2k-1)}j_{\sigma(2k)}'},
\end{align}
where $\mathbf{1}=(\bar{1},\bar{1},\cdots,\bar{1})\in \mathbb{Z}_2^k$.
\begin{theorem}\label{thm:pncenter}
The center of the universal enveloping algebra $\mathrm{U}(\mathfrak{p}_n)$ is trivial.
\end{theorem}
\begin{proof}
It is sufficient to prove that $\eta\circ\pi(\theta_{\sigma^{-1}})=-\eta\circ\pi(\theta_{\sigma^{-1}})$ for every $\sigma\in\mathfrak{S}_{2k}$ and $k\in\mathbb{N}$, since $\eta\circ\pi(\theta_{\sigma^{-1}})$ spans all the invariants in $S(\mathfrak{p}_n)$ and $S(\mathfrak{p}_n)$ is isomorphic to $\mathrm{U}(\mathfrak{p}_n)$ as $\mathfrak{p}_n$-modules. 

The relations of $S(\mathfrak{g})$ are 
\begin{align}
    G_{ij}&=-(-1)^{|j|(|i|+|j|)}G_{j'i'},\label{relationp} \\
    G_{ij}G_{kl}&=-(-1)^{(|i|+|j|)(|k|+|l|)}G_{kl}G_{ij},\label{relationSg}
\end{align}
 for all possible $i,j,k,l$. 
We have
\begin{align}\label{sgg}
\eta\circ\pi(\theta_{\sigma^{-1}})=\mathop{\sum}\limits_{\mathbf{J}\in W}
(-1)^{|\mathbf{J}^{o}|+|\mathbf{J}_{\sigma}^{o}|}p(\mathbf{1},\mathbf{J}_{\sigma}^o+\mathbf{J}_{\sigma}^e )\gamma(\mathbf{J},\sigma)G_{j_{\sigma(1)}j_{\sigma(2)}'} \cdots  G_{j_{\sigma(2k-1)}j_{\sigma(2k)}'},
\end{align}
by \eqref{pcenter}.

Recall that $K$ is the subgroup of $\mathfrak{S}_{2k}$ generated by the swaps $(12), (34),\cdots, (2k-1,2k)$. Take $\tau\in K$ and $g\in\mathfrak{S}_k$. The element $\tau$ indicates which terms transform by using relation \eqref{relationp}, and the element $g$ indicates how to transform by using relation \eqref{relationSg}. Specifically, we have
\begin{align}
    G_{j_{\sigma(1)}j_{\sigma(2)}'} \cdots  G_{j_{\sigma(2k-1)}j_{\sigma(2k)}'}=(-1)^{\clubsuit}G_{j_{\sigma\tau(1)}j_{\sigma\tau(2)}'} \cdots  G_{j_{\sigma\tau(2k-1)}j_{\sigma\tau(2k)}'},
\end{align}
where $\clubsuit=l(\tau)\mathop{\sum}\limits_{i\in \tau^>}|j_{\sigma\tau(i+1)}'|(|j_{\sigma\tau(i)}|+|j_{\sigma\tau(i+1)}'|)$, $l(\tau)$ is the length of permutation $\tau$, and $\tau^>=\{i |\tau(i)>i\}$. And
\begin{align}\label{ggg}
    G_{j_{\sigma\tau(1)}j_{\sigma\tau(2)}'} \cdots  G_{j_{\sigma\tau(2k-1)}j_{\sigma\tau(2k)}'}=\gamma(\mathbf{1}+\mathbf{J}_{\sigma\tau}^o+\mathbf{J}_{\sigma\tau}^e,g)G_{j_{\sigma\tau\overline{g}(1)}j_{\sigma\tau\overline{g}(2)}'} \cdots  G_{j_{\sigma\tau\overline{g}(2k-1)}j_{\sigma\tau\overline{g}(2k)}'}.
\end{align}
Note that 
\begin{align}\label{AAA}
    (-1)^{\clubsuit}=(-1)^{\mathop{\sum}\limits_{i\in \tau^>}|j_{\sigma\tau(i)}||j_{\sigma\tau(i+1)}| }\cdot (-1)^{\mathop{\sum}\limits_{i\in \tau^>}|j_{\sigma\tau(i)}|+|j_{\sigma\tau(i+1)}| }=\gamma(\mathbf{J}_{\sigma},\tau)(-1)^{|\mathbf{J}_{\sigma}^{o}|+|\mathbf{J}_{\sigma\tau}^{o}|}.
\end{align}
Therefore, by \eqref{sgg}-\eqref{AAA}
\begin{align}\label{transform}
\eta\circ\pi(\theta_{\sigma^{-1}})
=&\mathop{\sum}\limits_{\mathbf{J}\in W}
(-1)^{|\mathbf{J}^{o}|+|\mathbf{J}_{\sigma\tau}^{o}|} \gamma(\mathbf{J},\sigma)\gamma( \mathbf{J}_{\sigma},\tau)
p(\mathbf{1},\mathbf{J}_{\sigma}^o+\mathbf{J}_{\sigma}^e)\gamma(\mathbf{1}+\mathbf{J}_{\sigma\tau}^o+\mathbf{J}_{\sigma\tau}^e,g )\nonumber \\
&\quad\quad\cdot G_{j_{\sigma\tau\overline{g}(1)}j_{\sigma\tau\overline{g}(2)}'} \cdots  G_{j_{\sigma\tau\overline{g}(2k-1)}j_{\sigma\tau\overline{g}(2k)}'}.
\end{align}
On the other hand, we can replace $i$ by $\sigma\tau\overline{g}\sigma^{-1}(i)$ in \eqref{sgg} for all $i$ if $\sigma\tau\overline{g}\sigma^{-1} \in H$ since the stabilizer of $W$ in $\mathfrak{S}_{2k}$ is $H$. Therefore,
\begin{align}\label{replace}
\eta\circ\pi(\theta_{\sigma^{-1}})
=&\mathop{\sum}\limits_{\mathbf{J}\in W}
(-1)^{|\mathbf{J}_{\sigma\tau\overline{g}\sigma^{-1}}^{o}|+|\mathbf{J}_{\sigma\tau\overline{g}}^{o}|}
p(\mathbf{1},\mathbf{J}_{\sigma\tau\overline{g}}^{o}+\mathbf{J}_{\sigma\tau\overline{g}}^{e})\gamma( \mathbf{J}_{\sigma\tau\overline{g}\sigma^{-1}},\sigma)\nonumber\\
&\quad\quad\cdot G_{j_{\sigma\tau\overline{g}(1)}j_{\sigma\tau\overline{g}(2)}'} \cdots  G_{j_{\sigma\tau\overline{g}(2k-1)}j_{\sigma\tau\overline{g}(2k)}'}.
\end{align}
We have
\begin{align}\label{gammaJ}
\gamma(\mathbf{J}_{\sigma\tau\overline{g}\sigma^{-1}},\sigma)=&\gamma(\mathbf{J},\sigma\tau\overline{g})\gamma(\mathbf{J},\sigma\tau\overline{g}\sigma^{-1})=\gamma(\mathbf{J},\sigma\tau)\gamma(\mathbf{J}_{\sigma\tau},\overline{g})\gamma(\mathbf{J},\sigma\tau\overline{g}\sigma^{-1})\nonumber\\
=&\gamma(\mathbf{J},\sigma)\gamma(\mathbf{J}_{\sigma},\tau)\gamma( \mathbf{J}_{\sigma\tau},\overline{g})\gamma(\mathbf{J},\sigma\tau\overline{g}\sigma^{-1})
\end{align}
by \eqref{gamma}. 

Since $\sigma\tau\overline{g}\sigma^{-1}\in H$, there are unique $\tau_1\in K, g_1\in\mathfrak{S}_k$ such that $\sigma\tau\overline{g}\sigma^{-1}=\tau_1\overline{g_1}$, we have
\begin{align}\label{sgng_1}
\gamma(\mathbf{J},\sigma\tau\overline{g}\sigma^{-1})=\gamma(\mathbf{J},\tau_1\overline{g_1})=\gamma(\mathbf{J}_{\tau_1},\overline{g_1})\gamma(\mathbf{J},\tau_1)=\gamma(\mathbf{J}_{\tau_1},\overline{g_1})=\gamma(\mathbf{J}_{\tau_1}^o+\mathbf{J}_{\tau_1}^e,g_1) =\gamma(\mathbf{1},g_1)
\end{align}
and
\begin{align}
(-1)^{|\mathbf{J}^{o}|+|\mathbf{J}_{\sigma\tau\overline{g}\sigma^{-1}}^{o}|}=(-1)^{|\mathbf{J}^{o}|+|\mathbf{J}_{\tau_1\overline{g_1}}^{o}|}=(-1)^{|\mathbf{J}^{o}|+|\mathbf{J}_{\tau_1}^{o}|}=\mathrm{sgn}(\tau_1)
\end{align}
by \eqref{gamma}, $\mathbf{J}\in W$ and Lemma \ref{tildetau}.
We have
\begin{align}\label{usegammap}
\gamma(\mathbf{1}+\mathbf{J}_{\sigma\tau}^o+\mathbf{J}_{\sigma\tau}^e,g )=p(\mathbf{1},\mathbf{J}_{\sigma\tau}^{o}+\mathbf{J}_{\sigma\tau}^{e} )p(\mathbf{1},\mathbf{J}_{\sigma\tau\overline{g}}^{o}+\mathbf{J}_{\sigma\tau\overline{g}}^{e} )\gamma(\mathbf{1},g)\gamma(\mathbf{J}_{\sigma\tau}^o+\mathbf{J}_{\sigma\tau}^e,g)
\end{align}
by Lemma \ref{eq::gammaprelation}. Since,
\begin{align}\label{sgng}
\gamma(\mathbf{1},g)=\mathrm{sgn}(g),\quad  \gamma(\mathbf{1},g_1)=\mathrm{sgn}(g_1),\text{ and }
(-1)^{|\mathbf{J}_{\sigma\tau}^{o}|}=(-1)^{|\mathbf{J}_{\sigma\tau\overline{g}}^{o}|}.
\end{align}
We have
\begin{align}\label{overlineg}
\gamma(\mathbf{J}_{\sigma\tau},\overline{g})=\gamma(\mathbf{J}_{\sigma\tau}^o+\mathbf{J}_{\sigma\tau}^e,g)\text{ and }
p(\mathbf{1},\mathbf{J}_{\sigma\tau}^{o}+\mathbf{J}_{\sigma\tau}^{e} )=p(\mathbf{1},\mathbf{J}_{\sigma}^{o}+\mathbf{J}_{\sigma}^{e} )
\end{align}
by Lemma \ref{tildetau} and $\tau\in K$, respectively.

Substitute \eqref{gammaJ}-\eqref{overlineg} for \eqref{transform} and \eqref{replace}, we conclude that \eqref{transform} is equal to the negative of \eqref{replace} if and only if $$\mathrm{sgn}(\tau_1)\mathrm{sgn}(g)\mathrm{sgn}(g_1)=-1.$$
This holds due to the following key Lemma.
\end{proof}

\begin{keylemma}\label{symmetricgroup}
For every $\sigma\in \mathfrak{S}_{2k}$, there exists $\tau,\tau_1\in K$ and $g,g_1\in \mathfrak{S}_k$ such that
$(i)~~ \sigma\tau\overline{g}\sigma^{-1}=\tau_1\overline{g_1}$ and $(ii)~~\mathrm{sgn}(\tau_1)\mathrm{sgn}(g)\mathrm{sgn}(g_1)=-1$.
\end{keylemma}
Before proving the lemma, we introduce some notations.

An \textit{$(r,s)$-Brauer diagram} is a graph with $r+s$ dots, which are divided into pairs. It can be graphically represented by placing the $r$ dots on a horizontal line and the $s$ dots on a second horizontal line above the first one. The Brauer diagram consists of $(r + s)/2$ lines that connect the dots belonging to the same pair. We denote the set of all $(r,s)$-Brauer diagrams by $\mathcal{B}_r^s$. Hence, $\mathcal{B}_r^s$ is empty if and only if $r+s$ is odd, and if $r+s$ is even, then the cardinality of $\mathcal{B}_r^s$ is $(r+s-1)!!$. 

Recall that every permutation in symmetric group $\mathfrak{S}_k$ can be represented by a product of disjoint cycles. Add a new symbol $i'$ for each $1\leqslant i\leqslant k$ and define $[i]=[i']=i$ for each $i$, we call the \textit{kind} of $i,j'$ are \textit{different} and the \textit{kind} of $i,j$ or $i',j'$ are \textit{same} for all $1\leqslant i,j\leqslant k$. We call
$$(i_{1,1}i_{1,2}\cdots i_{1,j_1})(i_{2,1}i_{2,2}\cdots i_{2,j_2})\cdots(i_{s,1}i_{s,2}\cdots i_{s,j_s})$$
is \textit{a product of disjoint Brauer cycles} if $\{i_{r,t}\}_{1\leqslant r\leqslant s,1\leqslant t\leqslant j_s}$ is a subset of  $ \{1,2,\cdots, k\}\cup \{1',2',\cdots, k'\}$, $j_1+j_2\cdots+j_s=k$ and the set $\{[i_{r,t}] \}_{1\leqslant r\leqslant s,1\leqslant t\leqslant j_s}$ equal to $\{1,2,\cdots,k\}$.

For a product of disjoint Brauer cycles, we can draw a Brauer diagram as follows:
\begin{itemize}
    \item Step 1: Draw $2k$ dots on two horizontal lines, each of them consisting of $k$-dots and one above the other;
    \item Step 2: If the kind of $i_{1,1}$ and $i_{1,2}$ are the same (different), draw a line connecting the $[i_{1,1}]$-th dot in the lower line and $[i_{1,2}]$-th dot in the higher (lower) line;
    \item Step 3: If the kind of $i_{1,2}$ and $i_{1,3}$ are the same (different) and the $[i_{1,2}]$-th dot in the higher line has been connected, draw a line connecting the $[i_{1,2}]$-th dot in the lower line and the $[i_{1,3}]$-th dot in the higher (lower) line; if the kind of $i_{1,2}$ and $i_{1,3}$ are the same (different) and the $[i_{1,2}]$-th dot in the lower line has been connected, draw a line connecting the $[i_{1,2}]$-th dot in the higher line and the $[i_{1,3}]$-th dot in the lower (higher) line;
    \item Step 4: Repeat step 3 for $(i_{1,3},i_{1,4}),...,(i_{1,j_{1}-1},i_{1,j_{1}}), (i_{1,j_{1}},i_{1,1})$;
    \item Step 5: Repeat steps 2,3 and 4 for the remaining Brauer cycles.
\end{itemize}
For instance, the subsequent Brauer diagram can be represented by $(13)(24')(567')$ or $(6'75')(2'4)(31)$.

$$
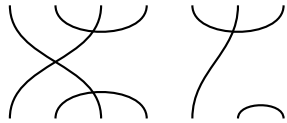
\begin{tikzpicture}[scale=1,thick,>=angle 90]
\begin{scope}[xshift=4cm]
\draw (2.8,1) to  [out=-90,in=90] +(1.2,-1.5);
\draw (3.4,1) to [out=-90, in=-90] +(1.2,0);
\draw (2.8,-0.5) to [out=90, in=-90] +(+1.2,+1.5);
\draw (3.4,-0.5) to [out=90, in=90] +(1.2,0);
\draw (5.2,1) to [out=-90, in=-90] +(1.2,0);
\draw (5.2,-0.5) to [out=90, in=-90] +(0.6,1.5);
\draw (5.8,-0.5) to [out=90, in=90] +(0.6,0);
\end{scope}
\end{tikzpicture}
$$

Every $(k,k)$-Brauer diagram can be represented by a product of disjoint Brauer cycles. Suppose the number of Brauer cycles with a length equal to $s$ is $\lambda_s$. We call the \textit{type} of this Brauer diagram $1^{\lambda_1}2^{\lambda_2}\cdots k^{\lambda_k}$. One can calculate that the number of products of disjoint Brauer cycles of a Brauer diagram with type $1^{\lambda_1}2^{\lambda_2}\cdots k^{\lambda_k}$ is $\mathop{\prod}\limits_{i=1}^n(2i)^{\lambda_i}\lambda_i!$ if we arrange the Brauer cycles in ascending order from smallest to largest. Therefore, the number of Brauer diagrams of type $1^{\lambda_1}2^{\lambda_2}\cdots k^{\lambda_k}$ is
$$\frac{2^kk!}{\mathop{\prod}\limits_{i=1}^n(2i)^{\lambda_i}\lambda_i!}.$$
The number of all $(k,k)$-Brauer diagrams is $(2k-1)!!$, so we get the equation
\begin{equation}\label{Brauerequation}
   \mathop{\sum}\limits_{\substack{\lambda_i\geqslant 0\\ \lambda_1+2\lambda_2+\cdots+k\lambda_k=k}} \frac{2^kk!}{\mathop{\prod}\limits_{i=1}^n(2i)^{\lambda_i}\lambda_i!}=(2k-1)!!.
\end{equation}

The $2k$-dots in a $(k,k)$-Brauer diagram are labeled as $1,2,\cdots,2k$ from top to bottom and left to right. We refer to the Brauer diagram $d_{\sigma}$ mentioned in Remark \ref{sigmaBrauer} as a $\sigma$\textit{-Brauer diagram}. Then, $d_{\sigma'}=d_{\sigma}$ if and only if their left cosets $\sigma H$ and $\sigma' H$ are the same. For example, the $\sigma=(267453)$-Brauer diagram $d_{\sigma}$ is labeled as
\begin{figure}[h]
\begin{equation}\label{fig:enter-label}
    \begin{tikzpicture}[scale=1,thick,>=angle 90]
\begin{scope}[xshift=4cm]
\draw (2.8,1) to  [out=-90,in=90] +(1.2,-1.5);
\node (node001) at (2.8,1.2)  {$1$};
\node (node002) at (2.8,-0.7)  {$2$};
\node (node003) at (3.4,1.2)  {$3$};
\node (node004) at (3.4,-0.7)  {$4$};
\node (node005) at (4,1.2)  {$5$};
\node (node006) at (4,-0.7)  {$6$};
\node (node007) at (4.6,1.2)  {$7$};
\node (node008) at (4.6,-0.7)  {$8$};
\draw (3.4,1) to [out=-90, in=-90] +(1.2,0);
\draw (2.8,-0.5) to [out=90, in=-90] +(+1.2,+1.5);
\draw (3.4,-0.5) to [out=90, in=90] +(1.2,0);
\end{scope}
\end{tikzpicture}
\end{equation}
\captionsetup{labelformat=empty}
  \caption{Labeled $(267453)$-Brauer diagram}
\end{figure}
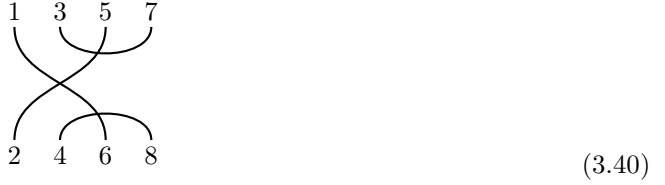

The symmetric group $\mathfrak{S}_{2k}$ acts on the $(k,k)$-Brauer diagram by $\sigma\cdot d_{\sigma'}=d_{\sigma\sigma'}$ for all $\sigma,\sigma'\in\mathfrak{S}_{2k}$. Recall that the subgroup $H$ and $d_{\sigma}=d_{\sigma'}$ if and only if $\sigma H=\sigma' H$. Therefore, the following statements hold: 
\begin{itemize}
\item[(I):] $g\in H$ if and only if  the action of $g$ on the $(k,k)$-Brauer diagram $d_{(1)}$ is invariant.
\item[(II):] $g\in \sigma H\sigma^{-1}$ if and only if the action of $g$ on the $(k,k)$-Brauer diagram $d_{\sigma}$ is invariant.

For a $\sigma$-Brauer diagram, we define its \textit{closure diagram} $\overline{d_{\sigma}}$ by adding $k$ red lines which connect the $s$-th dot in the lower line and the $s$-th dot in the higher line for all $1\leqslant s\leqslant k$. For example, the closure diagram of the $(267453)$-Brauer diagram in Figure \ref{fig:enter-label} is
\begin{figure}[h]
\begin{equation}\label{fig:closure}
\begin{tikzpicture}[scale=1,thick,>=angle 90]
\begin{scope}[xshift=4cm]
\draw (2.8,1) to  [out=-90,in=90] +(1.2,-1.5);
\node (node001) at (2.8,1.2)  {$1$};
\node (node002) at (2.8,-0.7)  {$2$};
\node (node003) at (3.4,1.2)  {$3$};
\node (node004) at (3.4,-0.7)  {$4$};
\node (node005) at (4,1.2)  {$5$};
\node (node006) at (4,-0.7)  {$6$};
\node (node007) at (4.6,1.2)  {$7$};
\node (node008) at (4.6,-0.7)  {$8$};
\draw (3.4,1) to [out=-90, in=-90] +(1.2,0);
\draw (2.8,-0.5) to [out=90, in=-90] +(+1.2,+1.5);
\draw (3.4,-0.5) to [out=90, in=90] +(1.2,0);
\draw[color=red] (2.8 ,1) to  [out=-90,in=90] +(0,-1.5);
\draw[color=red] (3.4 ,1) to  [out=-90,in=90] +(0,-1.5);
\draw[color=red] (4 ,1) to  [out=-90,in=90] +(0,-1.5);
\draw[color=red] (4.6 ,1) to  [out=-90,in=90] +(0,-1.5);
\end{scope}
\end{tikzpicture}
\end{equation}
\captionsetup{labelformat=empty}
  \caption{Closure diagram of $(267453)$-Brauer diagram}
\end{figure}
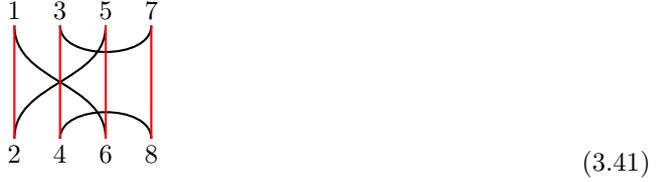
\item[(III):] $g\in \sigma H\sigma^{-1}\cap H$ if and only if the action of $g$ on the closure diagram $\overline{d_{\sigma}}$ is invariant.
\end{itemize}
\begin{remark}\label{symmetricclosure}
(1) The permutation $\sigma\in \mathfrak{S}_{2k}$ can be graphically represented by placing $2k$ dots on a horizontal line and $2k$ dots on a second horizontal line above the first one, and adding $2k$ lines. The $s$-th line connects the $s$-th dot in the lower line with the $\sigma(s)$-th dot in the higher line. For example, $\sigma=(267453)\in \mathfrak{S}_{8}$ can be represented as
$$
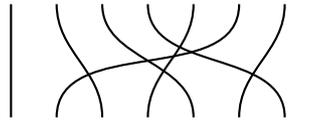
\begin{tikzpicture}[scale=1,thick,>=angle 90]
\begin{scope}[xshift=4cm]
\draw (2.8,-0.5) to  [out=90,in=-90] +(0,+1.5);
\draw (3.4,-0.5) to [out=90, in=-90] +(2.4,+1.5);
\draw (4,-0.5) to [out=90, in=-90] +(-0.6,+1.5);
\draw (4.6,-0.5) to [out=90, in=-90] +(0.6,+1.5);
\draw (5.2,-0.5) to [out=90, in=-90] +(-1.2,+1.5);
\draw (5.8,-0.5) to [out=90, in=-90] +(0.6,+1.5);
\draw (6.4,-0.5) to [out=90, in=-90] +(-1.8,+1.5);
\draw (7,-0.5) to [out=90, in=-90] +(0,+1.5);
\end{scope}
\end{tikzpicture}
$$
(2) The closure diagram $\overline{d_{\sigma}}$ with $\sigma\in \mathfrak{S}_{2k}$ can be also regarded as adding $k$ black arcs connecting the $(2s-1)$-th dot with the $(2s)$-th dot in the lower line and $k$ red arcs connecting the $(2s-1)$-th dot with the $(2s)$-th dot in the higher line. For example, the closure diagram of $\sigma=(267453)\in \mathfrak{S}_{8}$ can be represented as
$$
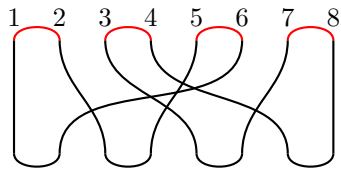
\begin{tikzpicture}[scale=1,thick,>=angle 90]
\begin{scope}[xshift=4cm]
\draw (2.8,-0.5) to  [out=-90,in=-90] +(0.6,0);
\draw (4,-0.5) to [out=-90, in=-90] +(0.6,0);
\draw (5.2,-0.5) to [out=-90, in=-90] +(0.6,0);
\draw (6.4,-0.5) to [out=-90, in=-90] +(0.6,0);
\draw (2.8,-0.5) to  [out=90,in=-90] +(0,+1.5);
\draw (3.4,-0.5) to [out=90, in=-90] +(2.4,+1.5);
\draw (4,-0.5) to [out=90, in=-90] +(-0.6,+1.5);
\draw (4.6,-0.5) to [out=90, in=-90] +(0.6,+1.5);
\draw (5.2,-0.5) to [out=90, in=-90] +(-1.2,+1.5);
\draw (5.8,-0.5) to [out=90, in=-90] +(0.6,+1.5);
\draw (6.4,-0.5) to [out=90, in=-90] +(-1.8,+1.5);
\draw (7,-0.5) to [out=90, in=-90] +(0,+1.5);
\draw[color=red] (2.8,1) to  [out=90,in=90] +(0.6,0);
\draw[color=red] (4,1) to [out=90, in=90] +(0.6,0);
\draw[color=red] (5.2,1) to [out=90, in=90] +(0.6,0);
\draw[color=red] (6.4,1) to [out=90, in=90] +(0.6,0);
\node (node001) at (2.8,1.3)  {$1$};
\node (node002) at (3.4,1.3)  {$2$};
\node (node003) at (4,1.3)  {$3$};
\node (node004) at (4.6,1.3)  {$4$};
\node (node005) at (5.2,1.3)  {$5$};
\node (node006) at (5.8,1.3)  {$6$};
\node (node007) at (6.4,1.3)  {$7$};
\node (node008) at (7,1.3)  {$8$};
\end{scope}
\end{tikzpicture}
$$
and it is same as Figure \ref{fig:closure}.
\end{remark}

We call two dots in the closure diagram \textit{connected} if they can be connected by  black lines and red lines, in addition, every line cannot be cut off. For example, in the closure diagram $\overline{d_{(267453)}}$, the dots $1,2,5,6$ are connected to each other, the dots $3,4,7,8$ are connected to each other, and the dots $1,3$ are not connected. One can observe that the dots of the component with number equal to $2s$ are $\lambda_s$ for all $s$, where $1^{\lambda_1}2^{\lambda_2}\cdots k^{\lambda_k}$ is the type of the Brauer diagram. 

Note that the dots in a component and lines connected to them form a \textit{circle} and the half of the dots in the component is the \textit{length} of the circle. Therefore, the closure diagram $\overline{d_{\sigma}}$ can be separated into $\lambda_i$ circles of length $i$. Recall (III) $g\in \sigma H\sigma^{-1}\cap H$ means that $g$ maps circles with length $i$ to circles with length $i$ for all $i$. Therefore, 
$$\sigma H\sigma^{-1}\cap H\cong \mathop{\prod}\limits_{i=1}^k (\underbrace{D_i\times D_i\times \cdots\times D_i}_{\lambda_i})\rtimes \mathfrak{S}_{\lambda_i},$$
where $D_i$ is the Dihedral group with order $2i$.

For example, the closure diagram $\overline{d_{(267453)}}$ can be separated into two circles as follows:
$$
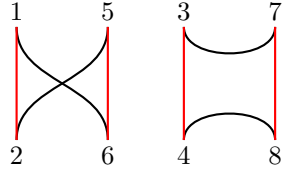
\begin{tikzpicture}[scale=1,thick,>=angle 90]
\begin{scope}[xshift=4cm]
\draw (2.8,1) to  [out=-90,in=90] +(1.2,-1.5);
\node (node001) at (2.8,1.2)  {$1$};
\node (node002) at (2.8,-0.7)  {$2$};
\node (node003) at (5,1.2)  {$3$};
\node (node004) at (5,-0.7)  {$4$};
\node (node005) at (4,1.2)  {$5$};
\node (node006) at (4,-0.7)  {$6$};
\node (node007) at (6.2,1.2)  {$7$};
\node (node008) at (6.2,-0.7)  {$8$};
\draw (5,1) to [out=-90, in=-90] +(1.2,0);
\draw (2.8,-0.5) to [out=90, in=-90] +(+1.2,+1.5);
\draw (5,-0.5) to [out=90, in=90] +(1.2,0);
\draw[color=red] (2.8 ,1) to  [out=-90,in=90] +(0,-1.5);
\draw[color=red] (5 ,1) to  [out=-90,in=90] +(0,-1.5);
\draw[color=red] (4 ,1) to  [out=-90,in=90] +(0,-1.5);
\draw[color=red] (6.2 ,1) to  [out=-90,in=90] +(0,-1.5);
\end{scope}
\end{tikzpicture}
$$
The order of group $ \sigma H\sigma^{-1}\cap H$ is $\mathop{\prod}\limits_{i=1}^k(2i)^{\lambda_i}\lambda_i!$. Hence,
\begin{align*}
    |H\sigma H|=|H\sigma H\sigma^{-1}|=\frac{|H|\cdot|\sigma H\sigma^{-1}|}{| H\cap \sigma H\sigma^{-1}|}=\frac{(2^kk!)^2}{\mathop{\prod}\limits_{i=1}^k(2i)^{\lambda_i}\lambda_i!}.
\end{align*}

Recall that the action of $H$ on the  $(0,2k)$-Brauer diagram appeared in Remark \ref{sigmaBrauer} is invariant. Hence, by Remark \ref{symmetricclosure}, if $\sigma'\in H\sigma H$, then the type of the Brauer diagram $d_{\sigma'}$ is the same as the type of the Brauer diagram $d_{\sigma}$. This means the double cosets $H\sigma H,~H\sigma'H$ are non-intersecting if the types of $d_{\sigma}$ and $d_{\sigma'}$ are different. Multiplying by $2^kk!$ on both sides of \eqref{Brauerequation}, we can get
\begin{equation*}
   \mathop{\sum}\limits_{\substack{\lambda_i\geqslant 0\\ \lambda_1+2\lambda_2+\cdots+k\lambda_k=k}} \frac{(2^kk!)^2}{\mathop{\prod}\limits_{i=1}^n(2i)^{\lambda_i}\lambda_i!}=(2k)!
\end{equation*}
and hence we establish the double coset decomposition of $\mathfrak{S}_{2k}$
\begin{align}\label{doublecoset}
\mathfrak{S}_{2k}=\mathop{\bigsqcup}\limits_{\substack{\lambda_i\geqslant 0\\ \lambda_1+2\lambda_2+\cdots+k\lambda_k=k} } H\sigma_{\boldsymbol{\lambda}}H,  
\end{align}
where $\boldsymbol{\lambda}=(\lambda_1,\lambda_2,\cdots,\lambda_k)$ and $\sigma_{\boldsymbol{\lambda}}$ is a permutation whose Brauer diagram is of type $1^{\lambda_1}2^{\lambda_2}\cdots k^{\lambda_k}$.

Now, we are prepared to prove Key Lemma \ref{symmetricgroup}.
\begin{proof}
(1) Suppose that the lemma holds for $\sigma$, then  the lemma is valid for every element in $\sigma H$.

The group $H$ is isomorphic to the semiproduct of $K$ and $\mathfrak{S}_{k}$. For every $h\in H$, there is a unique $\tau_2\in K$ and $g_2\in \mathfrak{S}_k$ such that $h=\tau_2\overline{g_2}$, hence 
$$ \tau_1\overline{g_1}=\sigma \tau\overline{g} \sigma^{-1}= (\sigma h) h^{-1}\tau h \cdot (\tau_2\overline{g_2})^{-1}\overline{g} (\tau_2 \overline{g_2})(\sigma h)^{-1}.  $$
There is a unique $\tau_3\in K$ such that $(\tau_2\overline{g_2})^{-1}\overline{g} (\tau_2 \overline{g_2})=\tau_3 \overline{g_2^{-1}gg_2}$. Let $\tau'= h^{-1}\tau h\tau_3$ and $g'=g_2^{-1}gg_2$. Therefore, $(\sigma h) \tau' \overline{g'} (\sigma h)^{-1}=\tau_1\overline{g_1}$ and $\mathrm{sgn}(\tau_1)\mathrm{sgn}(g')\mathrm{sgn}(g_1)=\mathrm{sgn}(\tau_1)\mathrm{sgn}(g)\mathrm{sgn}(g_1)=-1$.

(2) Suppose that the lemma holds for $\sigma$, then the lemma is valid for $\sigma^{-1}$.

Obviously, $\mathrm{sgn}(\overline{g})=1$ for all $g\in\mathfrak{S}_k$, so $\mathrm{sgn}(\tau)=\mathrm{sgn}(\tau_1)$. Therefore, $\sigma^{-1} \tau_1\overline{g_1} \sigma= \tau\overline{g}$ and $\mathrm{sgn}(\tau)\mathrm{sgn}(g_1)\mathrm{sgn}(g)=-1$.

By (1) and (2), we can conclude that if the lemma holds for $\sigma$, then the lemma holds for every element in $H\sigma H$. Therefore, by the double coset decomposition \eqref{doublecoset} of $\mathfrak{S}_{2k}$, we only need to prove for every type $1^{\lambda_1}2^{\lambda_2}\cdots k^{\lambda_k}$, there is $\sigma\in\mathfrak{S}_{2k}$ whose Brauer diagram is of type $1^{\lambda_1}2^{\lambda_2}\cdots k^{\lambda_k}$ such that the lemma holds.

Suppose that the type of $\sigma$-Brauer diagram is $1^{\lambda_1}2^{\lambda_2}\cdots k^{\lambda_k}$. A particular Brauer diagram of type $1^{\lambda_1}2^{\lambda_2}\cdots k^{\lambda_k}$ consists of $\lambda_l$ the following subgraph
$$
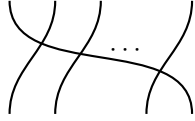
\begin{tikzpicture}[scale=1,thick,>=angle 90]
\begin{scope}[xshift=4cm]
\draw (2.8,1) to  [out=-90,in=90] +(2.4,-1.5);
\node (node005) at (4.35,0.35)  {$\cdots$};
\draw (3.4,-0.5) to [out=90, in=-90] +(0.6,1.5);
\draw (2.8,-0.5) to [out=90, in=-90] +(0.6,1.5);
\draw (4.6,-0.5) to [out=90, in=-90] +(0.6,1.5);
\end{scope}
\end{tikzpicture}
$$
where the number of lines in the subgraph is $l$. Take $\sigma'\in \mathfrak{S}_{2l}$ such that $d_{\sigma'}$ is the subgraph as above.

Suppose $\lambda_l>0$ and the lemma holds for $\sigma'$. Then, one can show that the lemma holds for $\sigma$ by extending the permutations in  $\mathfrak{S}_l$ and $\mathfrak{S}_{2l}$ into permutations in $\mathfrak{S}_k$ and $\mathfrak{S}_{2k}$.

Therefore, we only need to consider the particular Brauer diagram $d_{\sigma}$ as follows:
$$
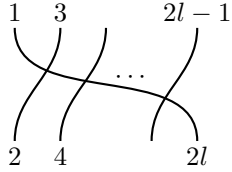
\begin{tikzpicture}[scale=1,thick,>=angle 90]
\begin{scope}[xshift=4cm]
\draw (2.8,1) to  [out=-90,in=90] +(2.4,-1.5);
\node (node001) at (2.8,1.2)  {$1$};
\node (node002) at (2.8,-0.7)  {$2$};
\node (node003) at (3.4,1.2)  {$3$};
\node (node004) at (3.4,-0.7)  {$4$};
\node (node005) at (4.35,0.35)  {$\cdots$};
\node (node007) at (5.2,1.2)  {$2l-1$};
\node (node008) at (5.2,-0.7)  {$2l$};
\draw (3.4,-0.5) to [out=90, in=-90] +(0.6,1.5);
\draw (2.8,-0.5) to [out=90, in=-90] +(0.6,1.5);
\draw (4.6,-0.5) to [out=90, in=-90] +(0.6,1.5);
\end{scope}
\end{tikzpicture}
$$
where $\sigma=(2l~ 2l-1~\cdots 3~2)$. This Brauer diagram can be represented as a cycle $(123\cdots l)$ and hence its type is $l$. Let $\tau=\tau_1=(12)(34)\cdots (2l-1~ 2l)$,

\begin{align*}
g=\begin{pmatrix}
    2 &3& \cdots&l-1& l\\
    l&l-1&\cdots &3&2
\end{pmatrix}\quad \text{and} \quad   
g_1=\begin{pmatrix}
    1&2& \cdots &l-1& l\\
    l&l-1&\cdots&2&1
\end{pmatrix},
\end{align*}
one can check that $\sigma \tau\overline{g} \sigma^{-1}=\tau_1\overline{g}_1$ and $\mathrm{sgn}(\tau_1)=(-1)^l, \mathrm{sgn}(g)=(-1)^{\frac{(l-1)(l-2)}{2}}$ and $\mathrm{sgn}(g_1)=(-1)^{\frac{l(l-1)}{2}}$ and hence the lemma holds for $\sigma$.
\end{proof}
\begin{remark}\label{tau1g1}
For $\sigma\in \mathfrak{S}_{2k}$ with $d_{\sigma}$ of type $k$, we can build all elements $\tau_2\overline{g_2}\in H$ satisfying Key Lemma \ref{symmetricgroup}. Here we prove this construction.

We take the closure diagram $\overline{d_{\sigma}}$, which consists of some circles. Let us say there is only a circle and its length is $k$. By choosing a direction,  we establish an oriented circle $i_1\rightarrow i_2\rightarrow\cdots\rightarrow i_{2k}\rightarrow i_1$. We refer the points $a$ and $b$ in the following subgraph as the entry point and the exit point, respectively.
$$
\begin{tikzpicture}[scale=1,thick,>=angle 90]
\begin{scope}[xshift=4cm]
\node (node001) at (3.3,1)[color=red]  {$>$};
\node (node002) at (4.3,1)  {$>$};
\node (node003) at (3.8,1.3)  {$a$};
\draw (2.8,1)[color=red] to  [out=0,in=180] +(1,0);
\draw (3.8,1) to  [out=0,in=180] +(1,0);
\node (node003) at (3.8,1)  {$\bullet$};
\node (node001) at (6.3,1)  {$>$};
\node (node002) at (7.3,1)[color=red]  {$>$};
\node (node003) at (6.8,1.3)  {$b$};
\draw (5.8,1) to  [out=0,in=180] +(1,0);
\draw (6.8,1)[color=red] to  [out=0,in=180] +(1,0);
\node (node003) at (6.8,1)  {$\bullet$};
\end{scope}
\end{tikzpicture}
$$
Therefore, $\tau_1\overline{g_1}$ can be chosen by the permutation
$$\begin{pmatrix}
    i_1& i_2&\cdots &i_k\\
    i_s& i_{s-1}&\cdots &i_{s+1}
\end{pmatrix}$$
where the one of $i_1, i_s$ is the entry point and the other is the exit point.
For example, let $\sigma=(23)(45)\in \mathfrak{S}_6$, an orientation of $\overline{d_{\sigma}}$ is
$$
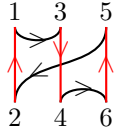
\begin{tikzpicture}[scale=1,thick,>=angle 90]
\begin{scope}[xshift=4cm]
\draw (2.8,1) to  [out=-90,in=-90] +(0.6,0);
\draw (3.4,0) to [out=90,in=90] +(0.6,0);
\draw (2.8,0) to [out=90,in=-90] +(1.2,1);
\draw (2.8,1)[color=red] to  [out=-90,in=90] +(0,-1);
\draw (3.4,1)[color=red] to  [out=-90,in=90] +(0,-1);
\draw (4,1)[color=red] to  [out=-90,in=90] +(0,-1);
\node (node001) at (2.8,1.2)  {$1$};
\node (node002) at (2.8,-0.2)  {$2$};
\node (node003) at (3.4,1.2)  {$3$};
\node (node004) at (3.4,-0.2)  {$4$};
\node (node007) at (4,1.2)  {$5$};
\node (node008) at (4,-0.2)  {$6$};
\node (node008) at (3.1,0.82)  {$>$};
\node (node008) at (3.7,0.18)  {$>$};
\node (node008) at (2.8,0.5)[color=red]  {$\wedge$};
\node (node008) at (3.4,0.7)[color=red]  {$\vee$};
\node (node008) at (4,0.5)[color=red]  {$\wedge$};
\node (node008)[rotate=25] at (3.1,0.4)  {$<$};
\end{scope}
\end{tikzpicture}
$$
Then, the oriented circle is $1\rightarrow 3\rightarrow 4\rightarrow 6\rightarrow 5\rightarrow2\rightarrow 1$, and $1,4,5$ are entry points, while $2,3,6$ are exit points.
$\tau_1\overline{g_1}$ is equal to one of three permutations
$$\begin{pmatrix}
    1&3&4&6&5&2\\
    2&5&6&4&3&1
\end{pmatrix},\quad\begin{pmatrix}
    1&3&4&6&5&2\\
    3&1&2&5&6&4
\end{pmatrix},\quad \begin{pmatrix}
    1&3&4&6&5&2\\
    6&4&3&1&2&5
\end{pmatrix},$$
one can calculate the $\tau\overline{g}$ and check that $\mathrm{sgn}(\tau_1)\mathrm{sgn}(g)\mathrm{sgn}(g_1)=-1$.

This process can be generalized to any $\sigma\in \mathfrak{S}_{2k}$. 
\end{remark}
\begin{remark}
 For $\sigma\in \mathfrak{S}_{2k}$, denote by $A^{\sigma}_1$ the set 
 $$\left\{~\tau_1\overline{g_1}~\left| ~\exists~ \tau,\tau_1\in K, g,g_1\in\mathfrak{S}_k,~~\mathrm{such ~that~ (i)~ and ~(ii)~ hold ~in~ Lemma ~\ref{symmetricgroup}.} \right.\right\}$$
 and $A^{\sigma}_{0}$ the complement of $A^{\sigma}_1$ in $H\cap \sigma H\sigma^{-1}$. One can show that $A_1^{\sigma}$ is a subgroup of $H\cap \sigma H\sigma^{-1}$ and
 $$A_{0}^{\sigma}A_1^{\sigma}=A_{1}^{\sigma}A_0^{\sigma}= A_1^{\sigma},\quad A_{1}^{\sigma}A_1^{\sigma}= A_0^{\sigma}.$$
Therefore, the cardinality of $A_{1}^{\sigma}$ and $A_0^{\sigma}$ are equal to half of the order of $H\cap \sigma H\sigma^{-1}$.

Using the notation in the above remark, the elements of $A_0^{\sigma}$ can be listed as follows:
$$\begin{pmatrix}
    i_1& i_2&\cdots &i_k\\
    i_s& i_{s+1}&\cdots &i_{s-1}
\end{pmatrix},$$
where the $i_1, i_s$ are either entry points or exit points. For example, if $\sigma=(23)(45)$, then
\begin{align*}
    A_0^{\sigma}=\left\{ \begin{pmatrix}
    1&3&4&6&5&2\\
    1&3&4&6&5&2
\end{pmatrix},~~\begin{pmatrix}
    1&3&4&6&5&2\\
    4&6&5&2&1&3
\end{pmatrix},~~\begin{pmatrix}
    1&3&4&6&5&2\\
    5&2&1&3&4&6
\end{pmatrix}  \right\}.
\end{align*}
\end{remark} 
\begin{theorem}\label{eta'surjective}
    The map $\eta': T(\mathfrak{g})^{\mathfrak{g}}\longrightarrow \mathcal{Z}(\mathfrak{g})$ is surjective for a Lie superalgebra listed in \eqref{eq:glist} with the exception of $\mathfrak{osp}_{2m|2n}$.
\end{theorem}

\begin{proof}
    This can be deduce from subsections \ref{subsectiongl}, \ref{subsectionq}, \ref{subsectionosp}, and \ref{subsectionp}.
\end{proof}

\begin{example}\label{ex:k3example}
Let $k=3$, by \eqref{standardinvariant}
\begin{align*}
c^{\otimes 3}=\mathop{\sum}\limits_{\mathbf{I}}(-1)^{|i_1|+|i_2|+|i_3|}e_{i_1}\otimes e_{i_1'}\otimes e_{i_2}\otimes e_{i_2'}\otimes e_{i_3}\otimes e_{i_3'}
=\mathop{\sum}\limits_{\mathbf{J}\in W}(-1)^{|\mathbf{J}^o|}e_{j_1}\otimes e_{j_2}\otimes \cdots \otimes e_{j_6},
\end{align*}
where $\mathbf{J}=(j_1,j_2,\cdots,j_6)=(i_1,i_1',i_2,i_2',i_3,i_3').$
Take $\sigma=(23)(45)\in \mathfrak{S}_6$, then
\begin{align*}
\gamma(\mathbf{J},\sigma)=&(-1)^{|j_2||j_3|+|j_4||j_5|}=(-1)^{|i_1'||i_2|+|i_2'||i_3|},\\
|\mathbf{J}_{\sigma}^o|=&(-1)^{|j_{\sigma(1)}|+|j_{\sigma(3)}|+|j_{\sigma(5)}|}=(-1)^{|i_1|+|i_1'|+|i_2'|},\\
p(\mathbf{1},\mathbf{J}_{\sigma}^o+\mathbf{J}_{\sigma}^e)=&(-1)^{2(|j_{\sigma(1)}|+|j_{\sigma(2)}|)+|j_{\sigma(3)}|+|j_{\sigma(4)}|}=(-1)^{|i_1'|+|i_3|}.
\end{align*}
Therefore, 
\begin{align}\label{example(23)(45)}
\eta\circ\pi(\theta_{\sigma^{-1}})=&\mathop{\sum}\limits_{\mathbf{J}}(-1)^{|\mathbf{J}^o|+|\mathbf{J}_{\sigma}^o|}p(\mathbf{1},\mathbf{J}_{\sigma}^o+\mathbf{J}_{\sigma}^e)\gamma(\mathbf{J},\sigma) G_{j_{\sigma(1)}j_{\sigma(2)}'}G_{j_{\sigma(3)}j_{\sigma(4)}'}G_{j_{\sigma(5)}j_{\sigma(6)}'}\nonumber\\
=&\mathop{\sum}\limits_{\mathbf{I}}(-1)^{|i_2'|+|i_3|+|i_1||i_2|+|i_2||i_3|} G_{i_1i_2'}G_{i_1'i_3'}G_{i_2'i_3}.
\end{align}
By Remark \ref{tau1g1}, let $\tau_1\overline{g_1}=\begin{pmatrix}
    1&3&4&6&5&2\\
    2&5&6&4&3&1
\end{pmatrix}$, so $\tau_1=(12)$ and $g_1=(23)$. There exist unique $\tau=(56)\in K, g=(12)\in\mathfrak{S}_3$ such that conditions (i) and (ii) hold in Key Lemma \ref{symmetricgroup}. The $\tau=(56)$ indicates that we firstly apply \eqref{relationp} in the third term $G_{i_2'i_3}$.
\eqref{example(23)(45)} can be rewritten as 
\begin{align}\label{examplecal1}
\mathop{\sum}\limits_{\mathbf{I}} -(-1)^{|i_3|(|i_2'|+|i_3|)}\cdot (-1)^{|i_2'|+|i_3|+|i_1||i_2|+|i_2||i_3|} G_{i_1i_2'}G_{i_1'i_3'}G_{i_3i_2'}.
\end{align}

The $g=(12)$ indicates that we secondly apply \eqref{relationSg} to change the first term $G_{i_1i_2'}$ and the second term $G_{i_1'i_3'}$ in \eqref{examplecal1}, then $\eta\circ\pi(\theta_{\sigma^{-1}})$ can be rewritten as 
\begin{align}\label{examplecal}
&\mathop{\sum}\limits_{\mathbf{I}} -(-1)^{|i_3|(|i_2'|+|i_3|)}\cdot(-1)^{(|i_1|+|i_2'|)(|i_1'|+|i_3'|)}\cdot (-1)^{|i_2'|+|i_3|+|i_1||i_2|+|i_2||i_3|} G_{i_1'i_3'}G_{i_1i_2'}G_{i_2'i_3}\nonumber\\
=&\mathop{\sum}\limits_{\mathbf{I}}(-1)^{|i_2|+|i_1||i_3|+|i_2||i_3|}G_{i_1'i_3'}G_{i_1i_2'}G_{i_2'i_3},
\end{align}
The permutation $\tau_1\overline{g_1}=\begin{pmatrix}
    1&2&3&4&5&6\\
    2&1&5&6&3&4
\end{pmatrix}$ indicates that we replace $(j_1,j_2\cdots,j_6)=(i_1,i_1',i_2,i_2',i_3,i_3')$ in \eqref{example(23)(45)} by $(j_{\tau_1\overline{g_1}(1)},j_{\tau_1\overline{g_1}(2)},\cdots,j_{\tau_1\overline{g_1}(6)})=(i_1',i_1,i_3,i_3',i_2,i_2')$, then we get
\begin{align*}
\eta\circ\pi(\theta_{\sigma^{-1}})=&\mathop{\sum}\limits_{\mathbf{I}}(-1)^{|i_3'|+|i_2|+|i_1'||i_3|+|i_3||i_2|}G_{i_1'i_3'}G_{i_1i_2'}G_{i_3'i_2}\\
=&-\mathop{\sum}\limits_{\mathbf{I}}(-1)^{|i_2|+|i_1||i_3|+|i_2||i_3|}G_{i_1'i_3'}G_{i_1i_2'}G_{i_3'i_2}.
\end{align*}
This formula is the negative of \eqref{examplecal}, thus $\eta\circ \pi(\theta_{\sigma}^{-1})=0$.
\end{example}
\begin{remark}
Let us briefly comment on the $\mathfrak{p}_n$-invariants in $T(\mathfrak{p}_n)$. According to Theorem \ref{thm:pncenter}, we know that the $\mathfrak{p}_n$-invariants in $S(\mathfrak{p}_n)$ and $\mathrm{U}(\mathfrak{p}_n)$ are trivial. However, this does not imply the absence of nontrivial $\mathfrak{p}_n$-invariants in $T(\mathfrak{p}_n)$. In fact, there exist numerous nontrivial $\mathfrak{p}_n$-invariants in $T(\mathfrak{p}_n)$, for example,
\begin{align*}
\pi(\theta_{\sigma^{-1}})=\mathop{\sum}\limits_{\mathbf{I}}(-1)^{|i_2'|+|i_3|+|i_1||i_2|+|i_2||i_3|} G_{i_1i_2'}\otimes  G_{i_1'i_3'}\otimes G_{i_2'i_3}
\end{align*}
is nonzero, where $\sigma=(23)(45)\in \mathfrak{S}_6$ as shown in Example \ref{ex:k3example}. Therefore, an interesting problem is to estimate the upper bound of the dimension of $\left(\mathfrak{g}^{\otimes k}\right)^{\mathfrak{g}}$ for classical Lie superalgebras $\mathfrak{g}$ listed in \eqref{eq:glist}. 
\end{remark}

The authors have no conflict of interest to declare that are relevant to this article.

\bibliographystyle{alpha}

\end{document}